\documentclass[a4paper,10pt]{amsart}
\usepackage{amsmath,amsthm,latexsym,amscd,amssymb,mathrsfs}
\usepackage[all]{xy}
\usepackage[hypertex,backref]{hyperref}
\numberwithin{equation}{section}

\newtheorem{prop}{Proposition}[section]
\newtheorem{lem}[prop]{Lemma}
\newtheorem{ddd}[prop]{Definition}
\newtheorem{theorem}[prop]{Theorem}
\newtheorem{cor}[prop]{Corollary}

\newcommand{\im}{\mathop{\mbox{\rm Im}}}

\newcommand{\vol}{\mathop{\mbox{\rm vol}}}

\newcommand{\dom}{\mathop{\rm dom}}
\newcommand{\spc}{\mathbf A}
\newcommand{\spcb}{\mathbf B}

\newcommand{\surg}{\mathcal S}
\newcommand{\ins}{\mathop{\it Inv}_{\alpha_0}\nolimits}

\newcommand{\Pj}{{\mathcal P}}

\newcommand{\T}{{\mathscr T}}
\newcommand{\Tl}{{\mathscr T}^{\times}}
\newcommand{\Tlp}{{\mathscr T}^{\times,s}}

\newcommand{\id}{\mathop{\rm id}\nolimits}

\newcommand{\Tr}{{\rm Tr}}

\newcommand{\spfl}{\mathop{\rm sf}\nolimits}

\newcommand{\Hi}{{\mathcal H}}
\newcommand{\inv}{\mathcal I}
\newcommand{\Y}{{\mathcal Y}}
\newcommand{\Yl}{{\mathcal Y}^{\times}}
\newcommand{\Ylp}{{\mathcal Y}^{\times,s}}
\newcommand{\N}{{\mathcal N}}
\newcommand{\Si}{{\mathcal S}}

\newcommand{\R}{{\mathcal R}}

\newcommand{\ideal}{{\mathcal I}}
\newcommand{\F}{{\mathcal F}}
\newcommand{\U}{{\mathcal U}}
\newcommand{\D}{{\mathcal D}}
\newcommand{\Li}{{\mathcal L}}

\newcommand{\Ca}{{\mathcal C}}
\newcommand{\A}{{\mathcal A}}
\newcommand{\B}{{\mathcal B}}
\newcommand{\Bi}{{\mathcal B}_{\infty}}
\newcommand{\Ai}{{\mathcal A}_{\infty}}

\newcommand{\C}{C^{\infty}}

\newcommand{\Ok}{\hat \Omega_k}

\newcommand{\Oi}{\hat \Omega_*}
\newcommand{\ra}{\partial}

\newcommand{\ten}{\otimes}

\newcommand{\pl}[1]{\varprojlim\limits_{#1}}

\newcommand{\ve}{\varepsilon}
\newcommand{\ov}{\overline}

\newcommand{\aj}{\dagger}

\DeclareMathOperator{\iu}{i}

\DeclareMathOperator{\supp}{supp}
\DeclareMathOperator{\Ran}{Ran}
\DeclareMathOperator{\sign}{sign}
\DeclareMathOperator{\ind}{ind}
\DeclareMathOperator{\ch}{ch}
\DeclareMathOperator{\Ker}{Ker}
\DeclareMathOperator{\di}{d}

\def\bbbr{{\rm I\!R}} 
\def\bbbn{{\rm I\!N}} 

\def\bbbc{{\rm I\!C}}

\def\bbbq{{\mathchoice {\setbox0=\hbox{$\displaystyle\rm Q$}\hbox{\raise
0.15\ht0\hbox to0pt{\kern0.4\wd0\vrule height0.8\ht0\hss}\box0}}
{\setbox0=\hbox{$\textstyle\rm Q$}\hbox{\raise
0.15\ht0\hbox to0pt{\kern0.4\wd0\vrule height0.8\ht0\hss}\box0}}
{\setbox0=\hbox{$\scriptstyle\rm Q$}\hbox{\raise
0.15\ht0\hbox to0pt{\kern0.4\wd0\vrule height0.7\ht0\hss}\box0}}
{\setbox0=\hbox{$\scriptscriptstyle\rm Q$}\hbox{\raise
0.15\ht0\hbox to0pt{\kern0.4\wd0\vrule height0.7\ht0\hss}\box0}}}}

\def\bbbz{{\mathchoice {\hbox{$\sf\textstyle Z\kern-0.4em Z$}}
{\hbox{$\sf\textstyle Z\kern-0.4em Z$}}
{\hbox{$\sf\scriptstyle Z\kern-0.3em Z$}}
{\hbox{$\sf\scriptscriptstyle Z\kern-0.2em Z$}}}}

\def\bbbc{{\mathchoice {\setbox0=\hbox{$\displaystyle\rm C$}\hbox{\hbox
to0pt{\kern0.4\wd0\vrule height0.9\ht0\hss}\box0}}
{\setbox0=\hbox{$\textstyle\rm C$}\hbox{\hbox
to0pt{\kern0.4\wd0\vrule height0.9\ht0\hss}\box0}}
{\setbox0=\hbox{$\scriptstyle\rm C$}\hbox{\hbox
to0pt{\kern0.4\wd0\vrule height0.9\ht0\hss}\box0}}
{\setbox0=\hbox{$\scriptscriptstyle\rm C$}\hbox{\hbox
to0pt{\kern0.4\wd0\vrule height0.9\ht0\hss}\box0}}}}

\title{Higher $\rho$-invariants\\ and the surgery structure set}
\author{Charlotte Wahl}

\begin{document}
\begin{abstract}
We study noncommutative $\eta$- and $\rho$-forms for homotopy equivalences. We prove a product formula for them and show that the $\rho$-forms are well-defined on the structure set. We also define an index theoretic map from $L$-theory to $C^*$-algebraic $K$-theory and show that it is compatible with the $\rho$-forms. Our approach, which is based on methods of Hilsum--Skandalis and Piazza--Schick, also yields a unified analytic proof of the homotopy invariance of the higher signature class and of the $L^2$-signature for manifolds with boundary.
\end{abstract}

\maketitle
\markright{{\textsc HIGHER $\rho$-INVARIANTS AND THE SURGERY STRUCTURE SET}}

\section{Introduction}

The study of metrics of positive scalar curvature and the study of differentiable structures on closed manifolds have inspired each other in a very fruitful way (see for example the survey \cite{rs}). Metrics of positive scalar curvature on a closed spin manifold $N$ represent elements in the set $\mathrm{Pos}^{\mathrm{spin}}_n(B\pi_1(N))$, $n=\dim N$, while differentiable structures on $N$ (not necessarily spin) can be investigated via the structure set ${\mathcal S}(N)$. 
An important tool in the study of differentiable structures is given by Wall's surgery exact sequence (see \cite[\S 5]{lu} for a survey), which has the form
$$L_{n+1}(\bbbz \pi_1(N)) \to {\mathcal S}(N) \to {\mathcal N}(N) \to L_n(\bbbz\pi_1(N)) \ .$$
The set $\mathrm{Pos}^{\mathrm{spin}}_n(B\pi_1(N))$ can be studied via  Stolz's sequence \cite{st}
$$R_{n+1}^{\mathrm{spin}}(B\pi_1(N)) \to \mathrm{Pos}^{\mathrm{spin}}_n(B\pi_1(N)) \to \Omega^{\mathrm{spin}}_n(B\pi_1(N)) \to R_{n}^{\mathrm{spin}}(B\pi_1(N)) \ .$$
Here we are only interested in the first map of each sequence: Roughly speaking, elements of the groups $L_{n+1}(\bbbz \pi_1(N))$ and $R_{n+1}^{\mathrm{spin}}(B\pi_1(N))$ can be represented by manifolds with boundary (with additional structures) and the maps $L_{n+1}(\bbbz \pi_1(N)) \to {\mathcal S}(N)$ and $R_{n+1}^{\mathrm{spin}}(B\pi_1(N)) \to \mathrm{Pos}^{\mathrm{spin}}_n(B\pi_1(N))$ are built from restriction to the boundary. 

For oriented $N$ information about the sets $L_{n+1}(\bbbz \pi_1(N))$ and ${\mathcal S}(N)$ can be obtained via
the $L^2$-signature $\sign_{(2)}$ and the $L^2$-$\rho$-invariant $\rho_{(2)}$ of Cheeger and Gromov. Namely, there is a commuting diagram \cite{ls,cw}
$$\xymatrix{
L_{n+1}(\bbbz \pi_1(N)) \ar[r] \ar[d]^{\sign-\sign_{(2)}} &{\mathcal S}(N) \ar[d]^{\rho_{(2)}}\\
\bbbr \ar[r]^{=}& \bbbr\ .}$$
 
A similar diagram exists for the positive scalar curvature sequence, with $\sign$ and $\sign_{(2)}$ replaced by the index and the $L^2$-index, respectively, of the spin Dirac operator on a manifold with boundary. However, in this case, the higher index theory for manifolds with boundary developed by Leichtnam and Piazza \cite{lpAPS,lptwisthigh} yields an even more general diagram \cite{lpposscal}. Recall that higher index theory, which generalizes $L^2$-index theory, studies Dirac operators twisted by the Mishenko--Fomenko bundle, which for $N$ is $\Pj:=\tilde N \times_{\pi_1(N)} C^*\pi_1(N)$. The index of such an operator is an element in $K_*(C^*\pi_1(N))$. For operators on manifolds with boundary, the index is only well-defined if the induced boundary operator is invertible. For invertible Dirac operators on closed manifolds there exists a higher $\eta$-invariant, which is a noncommutative differential form. The higher Atiyah--Patodi--Singer index theorem proven by Leichtnam and Piazza allows to show that the following diagram commutes. (To be precise, the framework of Leichtnam and Piazza uses the reduced $C^*$-algebra. The analogue for the maximal $C^*$-algebra follows from \cite{wazyl}.)
$$\xymatrix{
R_{n+1}^{\mathrm{spin}}(B\pi_1(N)) \ar[r]\ar[d]^{-\ind} & \mathrm{Pos}^{\mathrm{spin}}_n(B\pi_1(N))\ar[d]^{\rho}\\
K_*(C^*\pi_1(N))\ar[r]^-{\ch} &\Omega_*\Bi/\ov{\Omega_*^{<e>}\Bi + \dots}} $$ 
Here $\Omega_*\Bi$ is an algebra of noncommutative differential forms associated to the algebra $\Bi$, where $\Bi$ is a suitable Fr\'echet completion of $\bbbc\pi_1(N)$. The Karoubi Chern character $\ch$ is a generalization of the Chern character in differential geometry. The higher $\rho$-invariant is the ``delocalized'' part of the higher $\eta$-invariant, i. e. the part which survives after dividing out -- besides supercommutators and exact forms -- the subspace $\Omega_*^{<e>}\Bi$ of forms ``localized at the identity''. The map $\ind$ is the higher Atiyah--Patodi--Singer index mentioned before.  

The higher $\rho$-invariant for the signature operator twisted by $\Pj$ is in general not defined since the operator is in general not invertible. This is the reason why an analogue of the previous diagram for the surgery sequence is not straightforward. In this paper we will construct such an analogue: We define a map $\rho^{\mathcal S}:{\mathcal S}(N) \to  \Omega_*\Bi/\ov{\Omega_*^{<e>}\Bi + \dots}$ and an index theoretic map $\sign^L:L_{n+1}(\bbbz\pi_1(N)) \to K_*(C^*\pi_1(N))\ten\bbbz[\frac 12]$ and show that they are compatible with each other. Conjecturally, this latter map agrees (possibly up to a constant factor) with the well-known map defined using quadratic forms \cite{ro}.

Our result is related to the program ``Mapping surgery to analysis'', which was carried out using a different framework in \cite{hr}: Higson and Roe mapped the surgery exact sequence to an analytically defined exact sequence in $K$-theory. Since we deal with homology, heuristically our map $\rho^{\mathcal S}$ retains less information on the structure set, however the precise relation between both maps remains unclear. Already for the Atiyah--Patodi--Singer $\rho$-invariants the relation is complicated; it was established in \cite{hr2}. Piazza and Schick are studying an index theoretic approach to the results of Higson and Roe, which may clarify the relation. In this context they have developped independently some of the present methods \cite{pspers}.
The motivation for our approach lies in fact that (by pairing with cyclic cocycles) we obtain numerical invariants for which there is a product formula. In the case of positive scalar curvature such a product formula was proven and applied in \cite{lpposscal}.
 
We briefly sketch how $\rho^{\mathcal S}$ is constructed:
The definition of $\rho^{\mathcal S}$ is based on methods of Piazza and Schick \cite{ps} which in turn rely on a formalism of Hilsum and Skandalis \cite{hs}: Recall that an element in the structure set ${\mathcal S}(N)$ is represented by an oriented homotopy equivalence $f:M \to N$. Hilsum and Skandalis constructed an invertible selfadjoint Fredholm operator on $\Omega^*_{(2)}(N,\Pj)\oplus \Omega^*_{(2)}(M^{op},f^*\Pj)$ which can be connected to the signature operator by a  path of selfadjoint Fredholm operators. They concluded that the higher index of the signature operator on $N \cup M^{op}$ vanishes, i. e. that higher $K$-theoretic signatures are homotopy invariant. Piazza  and Schick showed that the invertible operator can be constructed as a perturbation of the signature operator by an integral operator. They used it to define an $L^2$-$\eta$-invariant. The generalization of their definition to the higher case is straightforward. We check that the noncommutative $\eta$-form thus obtained does not depend on the many choices involved in its definition and that the associated noncommutative $\rho$-form is in addition independent of the metric. 

Then we show that the $\rho$-form is well-defined on the structure set by generalizing the formalism of Hilsum and Skandalis to manifolds with cylindrical ends (which we consider instead of manifolds with boundary). Our approach also yields new proofs of the homotopy invariance of the higher signatures for manifolds with cylindrical ends, a result originally due to Leichtnam, Lott and Piazza \cite{llp}, and of the $L^2$-signatures \cite{ls}.  

We use the framework developed in \cite{wazyl}, which generalizes higher Atiyah--Patodi--Singer index theory: It allows for more general boundary conditions and more general $C^*$-algebras than the framework of \cite{lpAPS,lptwisthigh}. In particular, an essential tool will be the Atiyah--Patodi--Singer index theorem for Dirac operators over $C^*$-algebras (Theorem 9.4 in \cite{wazyl}, the higher analogue is Theorem 7.6 in \cite{lpAPS} and Theorem 4.1 in \cite{lptwisthigh}). 

{\bf Acknowlegments:} I am grateful to Paolo Piazza and Thomas Schick for information about their project and to Ulrich Bunke and Georges Skandalis for inspiring discussions. Furthermore I thank the referee for carefully reading the manuscript and pointing out some inaccurate technicalities.

\section{Preliminaries}
\label{setting}
In this and the following section we recall and adapt the relevant constructions of \cite{hs,ps}. 
Our results rely on a detailed understanding of them. In fact, we need to consider them for maps of degree one and not only for homotopy equivalences. 

We also need family versions of these constructions: We let the initial data depend on a parameter from a parameter space $B$. Later this will allow to conclude the independence of the construction of all choices.

Let $M$, $N$ be closed oriented manifolds of dimension $n$ and let $\tilde N \to N$ be a Galois covering with deck transformation group $\Gamma$. Let $\A$ be a unital $C^*$-algebra such that there is a unital involutive algebra homomorphism $\pi:\bbbc\Gamma \to \A$. The main examples we have in mind are $\A=C^*\Gamma$ or quotients thereof.

Let $M^{op}$ be $M$ with reversed orientation.

We define the $\A$-vector bundle $\F_N=\tilde N \times_{\Gamma} \A$ and endow it with the canonical $\A$-valued scalar product and the canonical flat hermitian connection, denoted by $d_N$. 

Let $B \subset \bbbr^m$ be a closed cube and let $f:B \times M \to N$ be a smooth map such that for each $b \in B$ the induced map $f_b:M \to N$ is a smooth map of degree one. We set $\F_M:=f^*_b\F_N$ for some $b\in B$ (and hence for all $b \in B$ up to a canonical isomorphism). We also pull back the $\A$-scalar product. The pull back of the hermitian connection $d_N$ is denoted by $d_M$. We write $\F$ for the bundle $\F_N\cup \F_M$ on $N\cup M$.

If $N$ is endowed with a Riemannian metric, then we write $\Omega_{(2)}^*(N, \F_N)=L^2(N,\Lambda^*T^*N \ten\F_N)$ and write $\Omega_{H^1}^*(N, \F_N)$ for the Sobolev space $H^1(N,\Lambda^*T^*N \ten\F_N)$.

Now we assume that $M$, $N$ are endowed with continuous families of Riemannian metrics parametrized by $B$. This defines Hilbert $C(B,\A)$-modules $C(B,\Omega_{(2)}^*(M,\F_M))$ and $C(B,\Omega_{(2)}^*(N, \F_N))$, as well as $C(B,\Omega_{H^1}^*(M,\F_M))$ and $C(B,\Omega_{H^1}^*(N,\F_N))$. Note that these are really spaces of sections since the spaces $\Omega_{(2)}^*(M,\F_M)$ and $\Omega_{(2)}^*(N, \F_N)$ depend on $b$. These spaces are Clifford modules with Clifford multiplication $c(\alpha)\omega =\alpha\wedge \omega-\iota(\alpha)\omega$. Clearly the Clifford multiplication and therefore also the chirality operators $\tau_M$, $\tau_N$ depend on $b\in B$. Recall that $\tau_M=\iu^{n/2}c(\vol_M)$ if $n$ is even and $\iu^{(n+1)/2}c(\vol_M)$ if $n$ is odd \cite[\S 3.2]{bgv}. Here $\vol$ denotes the volume form.

For the signature operator we use the conventions from \cite[\S 3.6]{bgv}. They are different for example from the ones in \cite{hs,ps}. 

Let $\D_M$ and $\D_N$ be the signature operators acting on $C(B,\Omega_{(2)}^*(M,\F_M))$ and $C(B,\Omega_{(2)}^*(N, \F_N))$, respectively. 

For $n$ even it holds that 
$$\D_M=d_M - \tau_M d_M \tau_M \ ,$$ 
and for $n$ odd 
$$\D_M=\tau_M d_M + d_M \tau_M \ .$$ 

The signature operator $\D=\left(\begin{array}{cc} \D_N & 0 \\ 0 & -\D_M  \end{array}\right)$ twisted by $\F$ on $N \cup M^{op}$ acts on the Hilbert $C(B,\A)$-module $$\Hi:=C(B,\Omega_{(2)}^*(N, \F_N))\oplus C(B,\Omega_{(2)}^*(M,\F_M)) \ .$$ 

Similarly we define 
$$\Hi^1:=C(B,\Omega_{H^1}^*(N, \F_N))\oplus C(B,\Omega_{H^1}^*(M,\F_M))\ .$$

We let $\tau$ be the chirality operator of $N \cup M^{op}$, thus 
$$\tau=\left(\begin{array}{cc} \tau_N & 0 \\ 0 & -\tau_M \end{array}\right) \ .$$
A quadratic form $Q$ on $\Hi$ is defined for $\alpha\in \Omega^k(N \cup M)$, $\beta \in \Omega^*(N \cup M)$, $v,w\in \C(N\cup M,\F)$ by
\begin{align*}
Q(\alpha \ten v,\beta \ten w)&:=Q_{N\cup M^{op}}(\alpha \ten v,\beta \ten w)\\
&:=(-1)^{kn+k(k-1)/2}\iu^{-(n+1)/2}\int_{N\cup M^{op}} (\alpha^* \wedge \beta) \langle v,w\rangle_{\F} \ .
\end{align*}
Similarly we have $Q_N$, $Q_M$, $Q_{M^{op}}$ etc.

By \cite[Prop. 3.58]{bgv} $$Q(\alpha,\tau \beta)=\langle \alpha,\beta \rangle \ .$$
The following properties are relevant for the formalism from \cite{hs}: 

If $n$ is even, the operator $D:=\left(\begin{array}{cc} d_N & 0 \\ 0 & d_M \end{array}\right)$ fulfills $Q(D\alpha,\beta)=-Q(\alpha,D\beta)$ and $D^2=0$. 

If $n$ is odd, this holds for the operator $D:=\iu\left(\begin{array}{cc} d_N & 0 \\ 0 & d_M \end{array}\right)$. 

We denote by $A^{\aj}:=\tau A^*\tau$ the adjoint of an operator $A$ with respect to $Q$. Thus $D^{\aj}=-D$.

Let $I=]-1,1[$. For $k \in 8\bbbn$ big enough there is a smooth map $p:B \times I^k \times M \to N$ such that for each $b \in B$ the induced map $p_b:I^k\times M \to N$ is a submersion and such that $p_b(0,x)=f_b(x)$ for each $x \in M$. We assume that $p$ extends to a smooth map from the closure $B \times \ov{I^k}$ to $N$. (We tacitly make this assumption in similar situations.) These properties ensure that the pullback $p^{\star}:=(p_b^{\star})_{b \in B}$ is bounded and adjointable as an operator from $C(B,\Omega_{(2)}^*(N, \F_N))$ to $C(B,\Omega_{(2)}^*(I^k \times M,\F_M))$. Thus we can define the pushforward $p_!:=(p^{\star})^{\aj}=\tau_N (p^{\star})^*\tau_{I^k \times M^{op}}=-\tau_N (p^{\star})^*\tau_{I^k \times M}$. (The adjoint $(p^{\star})^{\aj}$ is taken here with respect to $Q_N$ and $Q_{I^k \times M^{op}}$.)

Let $v \in C(B,\Omega_c^k(I^k))$ be a real-valued form with $\int_{I^k}v_b=1$ for all $b\in B$. Then the map
$$T_v(p):C(B,\Omega^*_{(2)}(N,\F_N)) \to C(B,\Omega^*_{(2)}(M,\F_M)),~ \omega \mapsto  \int_{I^k} v \wedge p^{\star}\omega$$
is a bounded adjointable operator between Hilbert $C(B,\A)$-modules. 
The operator is also bounded if we consider it as a map from $C(B,\Omega^*_{H^1}(N,\F_N))$ to $C(B,\Omega^*_{H^1}(M,\F_M))$. We sketch how this can be proven: Since $p$ is a submersion we may choose open sets $U_1 \subset I^k \times M$, $B' \subset B$, $U_2 \subset N$ and coordinate systems on $U_1$ and $U_2$ depending smoothly on $b\in B'$ such that $p_b:U_1 \to U_2$ is well-defined and represented by the restriction of a projection $\bbbr^{k+n} \to \bbbr^n$ for any $b \in B'$. We choose functions $\phi_j \in \C_c(U_j)$ such that $\phi_j$ equals $1$ on an open subset of $U_j$ and $\supp \phi_2$ is contained in the range of $p_b$ for all $b \in B'$. Using that the projection $\bbbr^{k+n} \to \bbbr^n$ induces a continuous map $H_{loc}^1(\bbbr^n) \to H^1_{loc}(\bbbr^{k+n})$, one checks that the operator $\phi_1p^{\star}\phi_2:C(B',\Omega^*_{H^1}(N,\F_N)) \to C(B',\Omega^*_{H^1}(I^k \times M,\F_M))$ is bounded. Now, using partitions of unity on $N, M, B$, we may conclude that $p^{\star}:C(B,\Omega^*_{H^1}(N,\F_N)) \to C(B,\Omega^*_{H^1}(I^k \times M,\F_M))$ is bounded. Furthermore, the embedding $C(B,\Omega^*_{H^1}(I^k\times M,\F_M)) \to C(B,\Omega^*_{(2)}(I^k)) \ten \Omega^*_{H^1}(M,\F_M)$ is bounded and $\int_{I^k} v\wedge: C(B,\Omega^*_{(2)}(I^k)) \to C(B)$ is bounded, as well. Note that all the maps involved are adjointable.

There is a stabilization process: Namely, $T_v(p)=T_{w\wedge v}(1\times p)$ for $1 \times p: I^l \times I^k \times M \to N$ with $w \in C(B,\Omega_c^l(I^l)), l \in 8\bbbn$ as above.

It holds that $T_v(p)(\dom d_N)\subset \dom d_M$ and $T_v(p)d_N=d_MT_v(p)$. Furthermore, $T_v(p)^{\aj}=-\tau_N T_v(p)^* \tau_M$.

Note that here our sign convention differs from the ones used in \cite{hs,ps} (which in general also do not agree with each other). The minus sign occurs since the adjoint $\aj$ is taken with respect to $Q_N$ on the source space and $Q_{M^{op}}$ on the target space of $T_v(p)$.  

Let the submersion $p$ depend on an additional parameter $t \in [0,1]$ such that $p: B \times \ov{I^k} \times M \times [0,1] \to N$ is smooth. Then, as operators from $C(B,\Omega^*(N,\F_N))$ to $C(B,\Omega^*(I^k \times M,\F_M))$
$$p_1^{\star}-p_0^{\star}= d_{I^k\times M} s_p+s_p d_N$$
with $s_p(\alpha)=\int_0^1 i_{\ra_t}p^{\star}\alpha$, see \cite[3.2. Lemme]{hs}. Here $p^{\star}$ is understood as a pullback from $C(B,\Omega_{(2)}^*(N,\F_N))$ to $C(B,\Omega_{(2)}^*(I^k \times M \times [0,1],\F_M))$.

Thus for $\omega \in C(B,\Omega^*(N,\F_N))$ we have in $C(B,\Omega^*(M,\F_M))$
\begin{align*}
(T_v(p_1)-T_v(p_0))(\omega) &= \int_{I^k} v \wedge (d_{I^k\times M} s_p + s_p d_N)(\omega)\\
&=d_M \int_{I^k} v \wedge s_p(\omega) + \int_{I^k} v \wedge s_p (d_N\omega) \ .
\end{align*}

The operator 
$$\omega \mapsto \int_{I^k} v \wedge s_p(\omega)$$
is bounded from $C(B,\Omega^*_{(2)}(N,\F_N))$ to $C(B,\Omega^*_{(2)}(M,\F_M))$. It is also bounded if we consider it as a map from $C(B,\Omega^*_{H^1}(N,\F_N))$ to $C(B,\Omega^*_{H^1}(M,\F_M))$. Furthermore, it maps $\dom d_N$ to $\dom d_M$. We say that $T_v(p_1)$ equals $T_v(p_0)$ modulo boundaries.

Note that any two submersions $p_0$, $p_1$ associated to $f$ as above are (after possibly stabilizing) homotopic through a path of submersions $p_t$ associated to $f$.

\begin{lem}
\label{iso}
If $f$ is a homotopy equivalence, then the map $T_v(p)$ induces an isomorphism from $\Ker d_N/\im d_N$ to $\Ker d_M/\im d_M$.
\end{lem}

\begin{proof}
For notational simplicity we suppress the dependence on $b$.

Let $g:N \to M$ be a smooth map such that $g \circ f$ and $f \circ g$ are homotopic to the identity and let $q:I^l \times N \to M$ be a submersion as above with $q(0,x)=g(x)$. Then, with $\id \times q:I^k \times(I^l \times N) \to I^{k}\times M$, 
\begin{align*}
T_{w}(q)T_v(p)\alpha&=\int_{I^{l}} w \wedge q^{\star}(\int_{I^k} v \wedge p^{\star} \alpha)\\
&=\int_{I^{k+l}} w\wedge v \wedge (\id \times q)^{\star}p^{\star}\alpha\\
&=T_{w \wedge v}(p \circ (\id \times q))\alpha \ .
\end{align*}

Since $(p \circ (\id \times q))(0,0,x)=(f\circ g)(x)$ and $f\circ g$ is homotopic to the identity, the submersion $p \circ (\id \times q)$ is homotopic to the identity through submersions (after possibly stabilizing). Thus $T_w(q)T_v(p)$ equals the identity modulo boundaries. Analogously, $T_v(p)T_w(q)$ equals the identity modulo boundaries. The proof is now straightforward.
\end{proof}

As before, in the following lemma we allow $f$ to be map of degree one.

\begin{lem}
\label{lemY}
Recall that $\dim M=n$. It holds:
\begin{enumerate}
\item There is a bounded operator $Y$ of degree $-1$ on $C(B,\Omega^*_{(2)}(N,\F_N))$ such that $$1+ T_v(p)^{\aj}T_v(p)=d_NY + Yd_N \ .$$ Furthermore, $(-1)^{n+1}Y^{\aj}=Y$, $Y(\dom d_N) \subset \dom d_N$, and $Y$ is bounded as an operator on $C(B,\Omega^*_{H^1}(N,\F_N))$.
\item Any convex linear combination of operators $Y$ fulfilling (1) fulfills (1).
\end{enumerate}
\end{lem}

\begin{proof}
For $i=1,2$ define $q_i:I^k \times I^k \times M\to I^k \times M,~(t_1,t_2,x) \mapsto (t_i,x)$.

For $\alpha,\beta \in \Omega^*_c(I^k)$, $\omega \in C(B,\Omega^*(M,\F_M))$ it holds, for example, $$q_{1!}(\alpha \wedge \beta \wedge \omega)=(\int_{I^k} \alpha) \wedge \beta \wedge \omega \ .$$

The adjoint of the map $$S_v:C(B,\Omega_{(2)}^*(I^k \times M,\F_M)) \to C(B,\Omega^*_{(2)}(M,\F_M)),~\alpha \mapsto \int_{I^k} v \wedge \alpha$$ 
equals $S_v^*(\beta) = \tau_{I^k} v \wedge  \beta$,
hence $S_v^{\aj}(\beta)=v \wedge \beta$.
Thus $$S_v^{\aj}S_v \alpha=v \wedge \int_{I^k}v \wedge \alpha=q_{2!}((q_1^{\star}v)\wedge (q_2^{\star}v) \wedge q_2^{\star}\alpha) \ .$$
Define on $C(B,\Omega_{(2)}^*(I^k \times M,\F_M))$
\begin{align*}
E_v: \alpha &\mapsto q_{2!}((q_1^{\star}v)\wedge (q_2^{\star}v) \wedge (q_1^{\star}\alpha))\\
& \quad =v \wedge \alpha\ .
\end{align*} 
Choose a homotopy of submersions $q:I^k \times I^k \times M \times [1,2] \to I^k \times M$ from $q_1$ to $q_2$ such that $q(I^k \times I^k \times \{x\}\times [1,2]) \subset I^k \times \{x\}$  for $x \in M$. Then for $\alpha \in C(B,\Omega^*_c(I^k \times M,\F_M))$
$$(S_v^{\aj}S_v-E_v)\alpha=q_{2!}\bigl((q_1^{\star}v)\wedge (q_2^{\star}v) \wedge (d_{I^k \times I^k \times M}s_{q}(\alpha)+s_{q}(d_{I^k \times M}\alpha))\bigr) \ .$$ 
On $C(B,\Omega^*_{(2)}(N,\F_N))$ define $E_v(p):\alpha \mapsto p_!(v \wedge p^{\star}\alpha)$. Since $T_v(p)=S_v p^{\star}$, we get 
\begin{align*}
\lefteqn{(T_v^{\aj}(p)T_v(p)-E_v(p))\alpha}\\
&=p_!(S_v^{\aj}S_v-E_v)p^{\star} \alpha \\
&=d_M (p_!q_{2!}\bigl((q_1^{\star}v)\wedge (q_2^{\star}v) \wedge s_{q}(p^{\star}\alpha)\bigr)+ p_!q_{2!}\bigl((q_1^{\star}v)\wedge (q_2^{\star}v) \wedge s_{q}(p^{\star} d_N\alpha)\bigr) \ .
\end{align*} 

A particular submersion $\ov p$ associated to $f$ can be constructed as follows: For $k\in 8\bbbn$ large there is an embedding $j:N \to \bbbr^k$. One may assume that $j(N) + \ov{I^k} \subset U$, where $U$ is a tubular neighbourhood of $j(N)$. It comes with a projection $\pi:U \to j(N)$. One defines (identifying $j(N)$ and $N$)
$$\ov{p}: I^k \times M \to N,~ \ov{p}(t,x)=\pi(f(x)+t) \ .$$

Since $p$ and $\ov{p}$ are homotopic through submersions (after possibly stabilizing), it holds  $E_v(\ov{p})=E_v(p)$ up to boundaries. 

It remains to study $E_v(\ov{p})$: For that aim define $\tilde \pi: I^k \times N \to N, \tilde \pi(t,x)=\pi(x+t)$. Thus $\ov{p}=\tilde \pi\circ (\id \times f)$. If $\alpha,\beta \in C(B,\Omega^*_c(I^k \times N,\F_N))$, then, since $f$ has degree one,
$$\int_{I^k \times M}(\id \times f)^{\star}\alpha \wedge (\id \times f)^{\star}\beta=\int_{I^k \times M}(\id \times f)^{\star}(\alpha \wedge \beta)=\int_{I^k \times N} \alpha \wedge \beta \ .$$
Thus $Q_{I^k \times M}((\id \times f)^{\star}\alpha,(\id \times f)^{\star}\beta)=Q_{I^k \times N}(\alpha,\beta)$. 
For $\alpha,\beta \in C(B,\Omega^*(N,\F_N))$ it follows that 
\begin{align*}
Q_N(\ov{p}_!(v \wedge \ov{p}^{\star}\alpha),\beta)&=Q_{I^k \times M^{op}}(v \wedge (\id\times f)^{\star} \tilde \pi^{\star}\alpha,(\id\times f)^{\star} \tilde \pi^{\star}\beta)\\
&=-Q_{I^k \times M}(v \wedge (\id\times f)^{\star} \tilde \pi^{\star}\alpha,(\id\times f)^{\star} \tilde \pi^{\star}\beta)\\
&=-Q_{I^k \times N}(v\wedge \tilde \pi^{\star}\alpha,\tilde\pi^{\star}\beta)\\
&=-Q_{N}(\tilde \pi_!(v \wedge \tilde\pi^{\star}\alpha),\beta) \ .
\end{align*}
Hence $E_v(\ov{p})=-E_v(\tilde\pi)$. Since (after possibly stabilizing) $\tilde \pi$ is homotopic to the projection $p_2:I^k \times N \to N$ and $E_v(p_2)=1$, the assertion follows. 

It remains to note that if $Y$ fulfills the equation in (1), then $\frac 12(Y + (-1)^{n+1}Y^{\aj})$ fulfills the equation as well.
\end{proof}

Up to this point the material was from \cite{hs}. Now we turn to \cite[\S 9]{ps}.

\begin{ddd}
An operator $B \in B(\Hi)$ is $\ve$-spectrally concentrated with respect to $\D$ if for any function $\psi \in \C_c(\bbbr)$ with $\supp \psi \cap [-\ve,\ve]=\emptyset$ it holds that $\psi(\D)B=B\psi(\D)=0$.

We say that $B$ is spectrally concentrated with respect to $\D$ if there is $\ve \in (0,\infty)$ such that $B$ is $\ve$-spectrally concentrated with respect to $\D$.
\end{ddd}

Clearly every spectrally concentrated operator is also bounded as an operator on $\Hi^1$. Furthermore, if $B$ is spectrally concentrated with respect to $\D$, then there is $\phi \in \C_c(\bbbr)$ with $\phi(\D)B=B\phi(\D)=B$. In particular, $B$ is smoothing since $\phi(\D)$ is a smoothing operator.

Let $(\phi_{\ve})_{\ve \in (0,\infty]}$ be a family of even functions in $\C(\bbbr)$ with values in $[0,1]$ such that $\supp \phi_{\ve} \subset (- \ve/2,\ve/2)$ and $\phi_{\ve}(x)=1$ for $x \in [-\ve/4,\ve/4]$. Note that $\phi_{\infty}=1$. Assume that the map $[0,\infty) \to \C_{loc}(\bbbr), ~t \mapsto \phi_{1/t}$ is smooth. 

An example can be constructed as follows: Let $\phi \in \C_c(\bbbr)$ be even such that $\supp \phi \subset (-1/2,1/2)$ and $\phi(x)=1$ for $x \in [-1/4,1/4]$. For $\ve>0$ set $\phi_{\ve}(x)=\phi(\frac{x}{\ve})$. (We need the flexibility of the general situation when we prove the independence of our definitions of the choice of $\phi_{\ve}$.)
  
Set 
$$\T_{v,\ve}(p)=\phi_{\ve}(\D_M)T_v(p) \phi_{\ve}(\D_N):C(B,\Omega^*_{(2)}(N,\F_N)) \to C(B,\Omega^*_{(2)}(M,\F_M)) \ .$$

For $\ve \in (0,\infty)$ the operator $\T_{v,\ve}(p)$ is a compact operator depending continuously on $\ve$ in the norm topology. For $\ve \to \infty$ it converges to $\T_{v,\infty}(p)=T_v(p)$ in the strong $*$-operator topology.

Thus for any compact subset $J \subset (0,\infty]$ the family $(\T_{v,\ve}(p))_{\ve \in J}$ defines an adjointable bounded operator from $C(J \times B,\Omega^*_{(2)}(N,\F_N))$ to $C(J\times B,\Omega^*_{(2)}(M,\F_M))$, for which we write $\T_v(p)$. Is also bounded and adjointable as an operator from 
$C(J \times B,\Omega^*_{H^1}(N,\F_N))$ to $C(J\times B,\Omega^*_{H^1}(M,\F_M))$. 

\begin{lem}
\label{lemYeps}
With $\dim M=n$ it holds:
\begin{enumerate}
\item $\T_{v,\ve}(p)(\dom d_N)\subset \dom d_M$; $\T_{v,\ve}(p)d_N=d_M\T_{v,\ve}(p)$.
\item If $f$ is a homotopy equivalence, then $\T_{v,\ve}(p)$ induces an isomorphism from $\Ker d_N/\im d_N$ to $\Ker d_M/\im d_M$.
\item There is a bounded operator $\Y$ of degree $-1$ on $C(J\times B,\Omega^*_{(2)}(N,\F_N))$ such that
$$1+\T_v(p)^{\aj}\T_v(p)=d_N\Y + \Y d_N \ .$$
Furthermore, $(-1)^{n+1}\Y^{\aj}=\Y$, $\Y(\dom d_N) \subset \dom d_N$, and $\Y$ is bounded as an operator on $C(J\times B,\Omega^*_{H^1}(N,\F_N))$.
\item The operator $\Y$ decomposes as $\Y' + \Y''$ such that $\Y'_{\ve}$ is $\ve$-spectrally concentrated for each $\ve \in J$ and $\Y''$ commutes with $\D_N^2$.
\item Any convex linear combination of operators $\Y$ fulfilling (3) and (4) fulfills (3) and (4).
\end{enumerate}
\end{lem}

\begin{proof}
We only give those arguments that will be needed later, and refer to \cite[Lemma 9.7]{ps} for the omitted details.

For $t \in [0,1]$ set $\phi_{t,\ve}:=(1-t)\phi_{\ve}+t$ and $g_{\ve}(x)=\frac{d}{dt}\phi_{t,\ve}(x)/x=(-\phi_{\ve}(x)+1)/x$. Since $(d_N+d_N^*)g(d_N+d_N^*)$ defines a bounded operator on $C(J \times B,\Omega_{(2)}^*(N,\F_N))$, the operator $g(d_N+d_N^*)$ is bounded from $C(J\times B,\Omega_{(2)}^*(N,\F_N))$ to $C(J \times B,\Omega_{H^1}^*(N,\F_N))$.

With 
\begin{align*}
Z_{\ve}&=\int_0^1 g_{\ve}(d_M+d_M^*) T_v(p) \phi_{t,\ve}(d_N+d_N^*)~dt\\
&\quad +\int_0^1 \phi_{t,\ve}(d_M+d_M^*) T_v(p) g_{\ve}(d_N+d_N^*)~dt
\end{align*}
it holds that
$T_v(p)-\T_v(p)=d_M Z + Z d_N$.

The integral converges in the strong-$*$-topology as an operator from $C(B,\Omega_{(2)}^*(N,\F_N))$ to $C( B,\Omega_{H^1}^*(M,\F_M))$. Thus we get a bounded operator $Z:C(J \times B,\Omega_{(2)}^*(N,\F_N)) \to C(J \times B,\Omega_{H^1}^*(M,\F_M))$. 

If $Y$ is as in the previous lemma, then the operator $\tilde Y=Y-Z^{\aj}T_v(p)+\T_{v}(p)^{\aj}Z$ fulfills the first equation in (3). In general it does not fulfill (4).

Set $\xi_{\ve}=(1-\phi_{2\ve}(x))\phi_{2\ve}(x)/x^2$ and $\zeta_{\ve}(x)=(1-\phi_{2\ve}(x))/x^2$.
Define 
$$Y^1_{\ve}:=\phi_{2\ve}(\D_N)\tilde Y_{\ve}\phi_{2\ve}(\D_N) \ , 
U^1_{\ve} =d_N^*\xi_{\ve}(\D_N) \ ,
U^2_{\ve} =d_N^*\zeta_{\ve}(\D_N) \ .$$
Note that $Y^1_{\ve}$ is $\ve$-spectrally concentrated and that $U^1_{\ve},~U^2_{\ve}$ commute with $\D_N^2$. Furthermore, $U^1,U^2$ are bounded as operators from $C(J \times B,\Omega_{(2)}^*(N,\F_N))$ to $C(J \times B,\Omega_{H^1}^*(N,\F_N))$.

Now set $\Y_{\ve}'=\frac 12(Y^1_{\ve}+(-1)^{n+1}(Y^1_{\ve})^{\aj})$ and $\Y_{\ve}''=\frac 12(U^1_{\ve}+U^2_{\ve}+ (-1)^{n+1}(U^1_{\ve}+U^2_{\ve})^{\aj})$. The arguments in \cite{ps} show that these operators fulfill the claim. 
\end{proof}

\section{The perturbed signature operator}
\label{signop}

The material of this section is essentially from \cite{hs,ps}. In the following we suppress the dependence on $p,v$ from the notation. 

Let $\gamma \omega =(-1)^k\omega$ for $\omega \in \Lambda^kT^*(N\cup M) \ten \F$.
  
For $\alpha \in \bbbr$, $\beta \in [0,1]$ we define on $C(J,\Hi)$

$$\R_{\beta}=\left(\begin{array}{cc} 1 & 0 \\  \beta \T & 1\end{array}\right),\quad 
\Li_{\alpha,\beta}=\left(\begin{array}{cc} 1+\beta^2 \T^{\aj}\T & (1- \iu\alpha  \gamma \Y)\beta \T^{\aj}\\ \beta \T(1 + \iu \alpha  \gamma \Y) & 1 \end{array}\right) \ .$$
If $\dim M=n$ is even, then we define
$$\delta_{\alpha}=\left(\begin{array}{cc} d_N & \iu \alpha \T^{\aj}\gamma \\ 0 &  d_M \end{array}\right) \ ,$$
and if $n$ is odd, we set
$$\delta_{\alpha}=\left(\begin{array}{cc} \iu d_N & -\alpha \T^{\aj}\gamma \\ 0 &  \iu d_M \end{array}\right) \ .$$

One checks that $\delta_{\alpha}^2=0$.

Set $$\Li_{\alpha,\beta}^e:=\Li_{0,\beta}=\left(\begin{array}{cc} 1+\beta^2 \T^{\aj}\T & \beta \T^{\aj}\\ \beta \T & 1 \end{array}\right)$$
and
$$\Li_{\alpha,\beta}^o:=\alpha\beta\left(\begin{array}{cc} 0 & - \iu \Y \T^{\aj}\\  \iu \T\Y & 0 \end{array}\right) \ .$$

Then $\Li_{\alpha,\beta}=\Li_{\alpha,\beta}^e + \gamma \Li_{\alpha,\beta}^o$ and $\gamma$ commutes with $\Li_{\alpha,\beta}^e$ and anticommutes with $\Li_{\alpha,\beta}^o$.

It will be important that the operator $\gamma$ does not appear in $\Li_{\alpha,\beta}^{e/o}$. This property does not hold for the operators defined in \cite{hs,ps}, which is why our definitions here are slightly different. 

The operators $\R_{\beta}$ and $\Li_{\alpha,\beta}$ are bounded. Since $\R_{\beta}$ is  invertible and $\R_{\beta}^{\aj}\R_{\beta}=\Li_{0,\beta}$, the latter operator is invertible for all $\beta$ and the operator $\Li_{\alpha,\beta}$ is invertible for $|\alpha| <\alpha_0$ with $\alpha_0$ small enough.

We define the set 
$$\ins:=\{(\alpha,\beta) \in [-\alpha_0,\alpha_0]\times [0,1]~|~ \beta=1 \mbox{ if } \alpha \neq 0 \} \ .$$

In the following we will always assume that $(\alpha,\beta) \in \ins$. 

Using $\gamma^{\aj}=(-1)^n\gamma$, one gets that $\Li_{\alpha,\beta}^{\aj}=\Li_{\alpha,\beta}$. Thus $$C_{\alpha,\beta}(\omega_1,\omega_2):=Q(\omega_1,\Li_{\alpha,\beta} \omega_2)$$ 
is a quadratic form in the sense of \cite{hs}. 

Using that $\beta=1$ for $\alpha\neq 0$, we get that
$\delta_{\alpha}^{\aj}\Li_{\alpha,\beta}=-\Li_{\alpha,\beta}\delta_{\alpha}$. Thus the operator $\delta_{\alpha}$ is skew-hermitian with respect to $C_{\alpha,\beta}$.

The selfadjoint operator $\Li_{\alpha,\beta}^{-1}\tau$ is compatible with $C_{\alpha,\beta}$ in the sense of \cite{hs}. Thus  $$\Si_{\alpha,\beta}:=(\Li_{\alpha,\beta}^{-1}\tau)|(\Li_{\alpha,\beta}^{-1}\tau)|^{-1}=\tau\Li_{\alpha,\beta} |\tau\Li_{\alpha,\beta}|^{-1}$$ 
is an involution compatible with $C_{\alpha,\beta}$. In the proof of the product formula it will become relevant that the operator $\tau$ cancels out from $|\tau\Li_{\alpha,\beta}|= (\Li_{\alpha,\beta}^*\Li_{\alpha,\beta})^{1/2}$. 

The operator $\delta_{\alpha} - \Si_{\alpha,\beta} \delta_{\alpha} \Si_{\alpha,\beta}$ is selfadjoint with respect to the scalar product 
\begin{align*}
\langle\omega_1,\omega_2\rangle_{\alpha,\beta}:&= C_{\alpha,\beta}(\omega_1,\Si_{\alpha,\beta}\omega_2)=\langle \omega_1, \tau \Li_{\alpha,\beta}\Si_{\alpha,\beta}\omega_2\rangle \\
&=\langle \omega_1, |\tau\Li_{\alpha,\beta}|\omega_2\rangle \ .
\end{align*}
Since $\delta_{\alpha}^2=0$, it holds that 
$$(\delta_{\alpha} - \Si_{\alpha,\beta} \delta_{\alpha} \Si_{\alpha,\beta})^2=-(\delta_{\alpha}+\Si_{\alpha,\beta}\delta_{\alpha}\Si_{\alpha,\beta})^2=-(\Si_{\alpha,\beta}\delta_{\alpha} + \delta_{\alpha}\Si_{\alpha,\beta})^2 \ .$$

If $n$ is even, we set 
$$\hat\D_{\alpha,\beta}=\delta_{\alpha} - \Si_{\alpha,\beta}\delta_{\alpha}\Si_{\alpha,\beta} \ ,$$ and in the odd-dimensional case we define 
$$\hat\D_{\alpha,\beta}:=(-\iu)(\Si_{\alpha,\beta}\delta_{\alpha} + \delta_{\alpha}\Si_{\alpha,\beta}) \ .$$ 
These operators are selfadjoint with respect to $\langle \cdot , \cdot \rangle_{\alpha,\beta}$.

In the proof of \cite[Prop. 1.9]{hs} it was shown that $\delta_{\alpha} - \Si_{\alpha,\beta} \delta_{\alpha} \Si_{\alpha,\beta}$ is invertible for $(\alpha,\beta) \in \ins$ with $\alpha \neq 0$ if $f$ is a homotopy equivalence. Thus, in this case the operator $\hat\D_{\alpha,\beta}$ is invertible.

The operator $\U_{\alpha,\beta}=|\tau\Li_{\alpha,\beta}|^{\frac 12}$ fulfills 
$$\langle\omega_1,\omega_2 \rangle_{\alpha,\beta}=\langle \U_{\alpha,\beta}\omega_1,\U_{\alpha,\beta}\omega_2 \rangle \ .$$
Thus $\tilde \D_{\alpha,\beta}:=\U_{\alpha,\beta}\hat \D_{\alpha,\beta}\U_{\alpha,\beta}^{-1}$ is selfadjoint with respect to the standard scalar product $\langle \cdot ,\cdot \rangle$. 

In the even-dimensional case $\tilde \D_{\alpha,\beta}$ anticommutes with the involution $\tilde \tau_{\alpha,\beta}:=\U_{\alpha,\beta}\Si_{\alpha,\beta}\U_{\alpha,\beta}^{-1}$.

We also set $\tilde \delta_{\alpha,\beta}:=\U_{\alpha,\beta}\delta_{\alpha}\U_{\alpha,\beta}^{-1}$.

With these definitions we get in the even case
$$\tilde \D_{\alpha,\beta}=\tilde \delta_{\alpha,\beta}-\tilde \tau_{\alpha,\beta} \tilde \delta_{\alpha,\beta}\tilde \tau_{\alpha,\beta} \ ,$$
and in the odd case
$$\tilde \D_{\alpha,\beta}=(-\iu)(\tilde \tau_{\alpha,\beta}\tilde \delta_{\alpha,\beta}+ \tilde \delta_{\alpha,\beta}\tilde \tau_{\alpha,\beta}) \ .$$

In the odd case we consider a different operator than the one defined in \cite[Def. 9.13]{ps}. For the relationship between both see the proof of Prop. \ref{L2prop}. The motivation for our choice is that our operator still looks formally like a signature operator.

It holds that $\tilde \tau_{0,0}=\tau$ and $\tilde \delta_{0,0}=D$. Thus $\tilde \D_{0,0}=\D$.

The operator $\tilde \D_{\alpha,\beta}$ is regular and selfadjoint as an unbounded operator on $C(J,\Hi)$ with domain $C(J,\Hi^1)$.
The operators $\Si_{\alpha,\beta}$, $\U_{\alpha,\beta}$, $\tilde \tau_{\alpha,\beta}$ depend continuously on $(\alpha,\beta) \in \ins$ in the norm topology as operators on $C(J,\Hi)$ as well as on the first Sobolev space $C(J,\Hi^1)$. The operator $\tilde \D_{\alpha,\beta}$ depends continuously on $(\alpha,\beta)  \in \ins$ in the norm topology as an operator from $C(J,\Hi^1)$ to $C(J,\Hi)$.

\begin{lem}
Assume that $K \in B(\Hi)$ is symmetric and $\ve$-spectrally concentrated. Then for any $f \in C(\bbbr)$ the operator $f(1+K)-f(1)$ is $\ve$-spectrally concentrated.
\end{lem}

\begin{proof} Since $K$ is bounded, we may assume that $f \in C_0(\bbbr)$.

If the assertion holds for $f,g \in C_0(\bbbr)$, then it clearly also holds for $f+g$. It also holds for $fg$ by 
$$f(1+K)g(1+K)-f(1)g(1)$$
$$=(f(1+K)-f(1))(g(1+K)-g(1))+ f(1)(g(1+K)-g(1)) + (f(1+K)-f(1))g(1) \ .$$ 
Furthermore, if $(f_n)_{n \in \bbbn}$ is a sequence in $C_0(\bbbr)$ converging uniformly to $f$ and the assertion holds for each $f_n$, then it also holds for $f$.

Since for any $\lambda \in \bbbr^*$ the algebra generated by the functions $(x \pm \lambda \iu)^{-1}$ is dense in $C_0(\bbbr)$, it is enough to show that $(1+K \pm \lambda \iu)^{-1}-(1 \pm  \lambda \iu)^{-1}$ are $\ve$-spectrally concentrated. For $\lambda$ large this follows from the Neumann series
$$(1+K \pm \lambda\iu)^{-1}-(1 \pm \lambda \iu)^{-1}= \sum_{n=1}^{\infty}(-1)^n((1 \pm \lambda \iu)^{-1}K)^n(1 \pm \lambda \iu)^{-1} \ .$$
\end{proof}

Recall that all operators we defined in this section act on $C(J,\Hi)$; thus they depend on $\ve \in J$. In the following lemma we specify $\ve$ and deal with operators acting on $\Hi$.

\begin{lem}
\label{smooth}
Let $\ve \in (0,\infty)$ and $(\alpha,\beta) \in \ins$.
\begin{enumerate}
\item
The operators $\Li_{\alpha,\beta,\ve}-1$ and $\delta_{\alpha,\ve}-D$ are $\ve$-spectrally concentrated. 
\item The operators $|\tau\Li_{\alpha,\beta,\ve}|-1,~|\tau \Li_{\alpha,\beta,\ve}|^{-1}-1,~\U_{\alpha,\beta,\ve}-1,~\U_{\alpha,\beta,\ve}^{-1}-1$ are $\ve$-spectrally concentrated.
\item The operators $A_{\alpha,\beta,\ve}:=\tilde \D_{\alpha,\beta,\ve}-\D$ and $\tilde \tau_{\alpha,\beta,\ve}-\tau$ are $\ve$-spectrally concentrated.
\end{enumerate}
\end{lem}

We will also write $A(f)_{\alpha,\beta,\ve}$ for $A_{\alpha,\beta,\ve}$ if we want to stress the dependence on $f$.

\begin{proof}
Assertion (1) follows from the fact that $\T_{\ve},~\T_{\ve}^{\aj}$ are $\ve$-spectrally concentrated and that $\Y_{\ve}$ is the sum of an $\ve$-spectrally concentrated operator and an operator commuting with $\D^2$.

Assertion (1) implies that $\tau\Li_{\alpha,\beta,\ve}-\tau$ is $\ve$-spectrally concentrated, and hence
$$(\tau\Li_{\alpha,\beta,\ve})^2-1=(\tau\Li_{\alpha,\beta,\ve}-\tau)^2+\tau(\Li_{\alpha,\beta,\ve}-1)\tau+(\Li_{\alpha,\beta,\ve}-1)$$ is $\ve$-spectrally concentrated as well. Now (2) follows from the previous lemma.

Using this one easily checks (3).
\end{proof}

\section{Noncommutative $\rho$-forms for homotopy equivalences}
\label{rhodef}

In this paper we only study $\rho$-forms for $\dim M =n$ odd. Probably our methods can be adapted to the even case. Alternatively, in the even case one may apply our results to $M \times S^1$ in order to get information on $M$.

Our definitions in this section are based on the framework exposed in \cite{wazyl}. 

We assume that $B$ is a point and that $f: M \to N$ is a homotopy equivalence.

Let $(\A_j,\iota_{j+1,j}:\A_{j+1} \to \A_j)_{j \in \bbbn_0}$ be a projective system of
involutive Banach algebras with unit satisfying the following conditions:
\begin{itemize} 
\item It holds that $\A_0=\A$.
\item The map $\iota_{j+1,j}:\A_{j+1} \to \A_j, j \in \bbbn_0$ is injective.
\item The induced map $\iota_j:\Ai:= \pl{i}\A_i \to \A_j, j \in \bbbn_0$ has dense
range.
\item The algebra $\A_j, j \in \bbbn$ is stable with respect to the holomorphic functional calculus in
$\A$.  
\item It holds that $\pi(\bbbc\Gamma) \subset \A_j, j \in \bbbn_0$.
\end{itemize}

For the last condition we consider $\A_j$ as a subalgebra of $\A$.

The projective limit $\Ai$ is an involutive  locally
$m$-convex Fr\'echet algebra with unit and $\pi(\bbbc\Gamma) \subset \Ai$.  

The universal algebra of differential forms is the $\bbbz$-graded space $\Oi\Ai= \prod_{k=0}^{\infty} \Ok \Ai$ with $\Ok \Ai := \Ai (\ten \Ai/\bbbc)^k, ~k \in \bbbn_0$. The tensor products here are completed projective. The differential $\di$ of degree one is given by 
$\di(a_0 \ten a_1 \dots \ten a_k)=1 \ten a_0 \ten a_1 \dots \ten a_k$ and linear extension. The product is determined by
$a_0 \ten a_1 \dots \ten a_k= a_0 (\di a_1) \dots (\di a_k)$, graded Leibniz rule and $\bbbc$-linearity. 

The involution on $\Ai$ extends to an involution on $\Oi\Ai$. Let $\ideal \subset \Oi\Ai$ be a closed homogeneous ideal, which is closed under the involution and under $\di$. We set $\Oi^{\ideal}\Ai:=\Oi\Ai/\ideal$. This is an involutive differential algebra. 

Most of the time we omit $\ideal$ from the notation for simplicity. 

We also recall that there is a Chern character $\ch: K_*(\A) \to H_*^{dR}(\Ai)$ with values in the de Rham homology associated to $\Oi^{\ideal}\Ai$.

We denote by $\Oi^{<e>}\Ai$ the differential algebra of forms localized at the identity. This is the closure of the subalgebra of $\Oi\Ai$ generated by the forms $\pi(g_0) \di \pi(g_1) \dots \di \pi(g_k)$ with $g_0, g_1, \dots ,g_k \in \Gamma$ and $\pi(g_0g_1 \dots g_k)=\pi(e)$.

For $j\in \bbbn \cup \infty$ we define $\F^j_N:=\tilde N \times_{\Gamma} \A_j$, $\F^j_M:=f^* \F^j_N$ and $\F^j:=\F_N^j \cup \F_M^j$.

Let $(U_k)_{k=1, \dots, m}$ be a finite open cover of $N$. Let $q:\tilde N \to N$ be the projection. We assume that for every $k$ there is a cross-section $\psi_k:U_k \to q^{-1}(U_k)$ and set $U_k'=\psi_k(U_k)$. Choose a smooth partition of unity $(\chi_k^2)_{k=1, \dots, m}$ subordinate to $(U_k)_{k=1, \dots, m}$. We write $g_{kl}: U_k \cap U_l \to \Gamma$ for the locally constant function whose image at a point $x$ is the deck transformation sending $U_l'\cap q^{-1}(x)$ to $U_k' \cap q^{-1}(x)$. Then $g_{kl}g_{lm}=g_{km}$ and $g_{kk}=1$ on the set where these expressions are defined. 

By $(P_N)_{kl}=\chi_k\chi_l g_{kl}$ a projection $P_N \in \C(N,M_m(\bbbc\Gamma))$ is defined. We set $P_M=P_N \circ f$.

The map $\F^j_N \to P_N(\A_j^m \times N)$ defined by $$\C(\tilde N,\A_j)^{\Gamma} \to P_N(\C(N,\A_j^m)),~s\mapsto (\chi_1 (s\circ \psi_1), \chi_2 (s\circ \psi_2), \dots, \chi_m (s\circ \psi_m))$$ is an isomorphism of $\A_j$-vector bundles. By pulling back the connection $P_N\di P_N$ we get a connection on $\F^j_N$ in direction of $\A_j$  (see \cite[\S 1.2]{wazyl} for this notion). Analogously $P_M\di P_M$ defines a connection on $\F^j_M$ in direction of $\A_j$.

Since $f$ is a homotopy equivalence, the operator $\D+A(f)_{\alpha,1,\ve}$ from the previous section is invertible.

Using the noncommutative connections, one gets an $\eta$-form $\eta(\D,A_{\alpha,1,\ve})$ for $(\alpha,1) \in \ins$ with $\alpha \neq 0$ and $\ve \in (0,\infty)$, see \S 5 and the remarks after Def. 9.5 in \cite{wazyl}. It is well-defined as an element of $\Oi\Ai/\ov{[\Oi\Ai,\Oi\Ai]_s + \di\Oi\Ai}$. 
(Here $[\cdot, \cdot]_s$ denotes the graded commutator.)

We set
$$\eta(f):=\frac 12\bigl(\eta(\D, A(f)_{\alpha,1,\ve})+\eta(\D,A(f)_{-\alpha,1,\ve})\bigr) \ .$$

The definition of $\eta(f)$ involves many choices: The definition of operator $A(f)_{\alpha,1,\ve}$ depends on the choice of the submersion $p$, the form $v$, the operator $\Y$, the function $\phi_{\ve}$ and the numbers $\alpha,\ve$. Our considerations so far imply that any two of these choices can be joined by a path of possible choices leading to a path $(A(f)_t)_{t \in [0,1]}$ of selfadjoint operators with smooth integral kernel such that $D_t:=\D + A(f)_t$ is invertible for all $t \in [0,1]$. It follows from the results of the first two sections that the family $(A(f)_t)_{t \in [0,1]}$ defines a bounded operator on $C([0,1],\Hi)$. 

We recall some facts on the noncommutative spectral flow proven in \cite{wancsf} (for the concept see also \cite{lpdir}). Let $B$ be a compact space and $(D_b)_{b\in B}$ a family of regular selfadjoint operators on $\Hi$ with domain $\Hi^1$ (here $\Hi^1$, $\Hi$ are as before) and such that $D_b:\Hi^1 \to \Hi$ is bounded and depends in a strongly continuous way on $b$. We assume that the family $(D_b)_{b\in B}$ defines a regular selfadjoint operator on the Hilbert $C(B)$-module $C(B,\Hi)$ with domain $C(B,\Hi^1)$. Thus $(D_b+\iu)_{b \in B}$ is invertible as a bounded operator from $C(B,\Hi^1)$ to $C(B,\Hi)$. Since multiplication by $\iu$ is compact as an operator between these spaces, it follows that the family $(D_b)_{b \in B}$ defines a Fredholm operator between them. 

In particular, if $B=[b_0,b_1]$ is an interval and if $D_b$ is invertible for $b=b_0,b_1$, then the spectral flow of $(D_b)_{b \in B}$ is well defined. If $(D_b)_{b\in B}$ is invertible as a regular operator on $C(B,\Hi)$, then the spectral flow vanishes. Futhermore if $D_b:\Hi^1 \to \Hi$ depends continuously on $b$ (in the operator norm topology), then $\spfl((D_b)_{b \in B})=\ind(\ra_b+D_b)$, see the remarks before \cite[Lemma 2.5]{aw}. Here $\ra_b+D_b$ is considered as an operator on $C_0(\bbbr,\Hi)$; we set $D_b:=D_{b_0}$ for $b<b_0$ and $D_b:=D_{b_1}$ for $b>b_1$.

Thus, for $D_t:=\D + A(f)_t$ as above, $\spfl((D_t)_{t\in [0,1]})=0$. Hence also $\ind(\ra_t - D_t)=0$.  The Atiyah--Patodi--Singer index theorem for Dirac operators over $C^*$-algebras (Theorem 9.4 in \cite{wazyl}) applied to $\ra_t - D_t$ yields
\begin{align*} 
0=\ch(\ind(\ra_t - D_t))
&=(2\pi \iu)^{-\frac{n+1}{2}}\int_{[0,1]\times (N\cup M)}L([0,1]\times (N\cup M))\ch(\F) \\
& \qquad +\eta(\D,A(f)_0)-\eta(\D,A(f)_1)\\
&= \eta(\D,A(f)_0)-\eta(\D,A(f)_1) \ .
\end{align*}
Here $L(~\cdot~)$ denotes the Hirzebruch class. The last equality follows from the fact that $\int_{[0,1]}L([0,1]\times (N\cup M))\ch(\F)$ vanishes.

Thus $\eta(f)$ does not depend on the choices. 

Note that $\eta(f)$ depends on the metrics on $M$, $N$. It also depends on the choice of the projection $P_N$. Only the degree zero part is independent of this latter choice. 

If we want to emphasize the dependence on the metric we write $\eta(f,g_M,g_N)$. If $M=N$ as Riemannian manifolds, then we may also write $\eta(f,g_M)$. 

Two different metrics $g_0$, $g_1$ on $N \cup M$ can be joined by a path of metrics $g_t$ giving rise to a path of invertible operators $\D_{g_t}+A(f)_t$. (Here the path $A(f)_t$ is of course different from the above one.) Here $\D_{g_t}$ is the signature operator with respect $g_t$. Then $\ind(\ra_t-\D_{g_t}-A(f)_t)$ vanishes in $K_0(\A)$. The Atiyah--Patodi--Singer index  theorem for Dirac operators over $C^*$-algebras applied to the operator $\ra_t- \D_{g_t}-A(f)_t$ implies that 
$$\eta(f,g_0|_M,g_0|_N)-\eta(f,g_1|_M,g_1|_N)$$
$$\in \Oi^{<e>}\Ai/\left(\ov{[\Oi\Ai,\Oi\Ai]_s + \di\Oi\Ai} \cap \Oi^{<e>}\Ai\right) \ ,$$ since the local contribution is an element in that space.

\begin{ddd}
We define the $\rho$-form 
$$\rho(f)=[\eta(f)] \in \Oi\Ai/\ov{[\Oi\Ai,\Oi\Ai]_s + \di\Oi\Ai+\Oi^{<e>}\Ai} \ .$$ 
\end{ddd}

By the previous arguments the $\rho$-form is independent of the metrics on $M$ and $N$. By the arguments in the proof of \cite[Lemma 12.1]{wazyl} it is also independent of the choice of the projection $P_N$.

\section{Relation to well-known invariants}

In the following we relate the $\eta$- and $\rho$-forms $\eta(f)$ and $\rho(f)$ from the previous section to well-known $\rho$-invariants.

\subsection{The higher case}
\label{high}

A noncommutative $\eta$-form $\eta(M,g_M)$ for the signature operator on a single manifold $M$ can be defined using the symmetric spectral sections introduced by Leichtnam--Piazza, see \cite{lpsign,lpdir}. (A different but equivalent regularization was given in \cite{llp}.) Symmetric spectral sections do not always exist. The corresponding analytic $\rho$-forms appear in \cite{lpsign} for groups of polynomial growth and in \cite{wazyl} in general. A topological analogue of these analytic higher $\rho$-form has been defined in \cite{we}.

Assume that the range of the closure of $d_M:\Omega^{m-1}(M,\F_M)\to \Omega_{(2)}^m(M, \F_M)$ with $m=(\dim M +1)/2$ is closed.

This assumption is rather strong; see \cite[\S 3]{llk} for equivalent conditions and \cite[Lemma 2.2]{llp} for examples in the case $\A=C^*_r\Gamma$.

\begin{lem}
The range of the closure of $d_N:\Omega^{m-1}(N,\F_N)\to \Omega_{(2)}^m(N, \F_N)$ is closed as well.
\end{lem}

\begin{proof}
Let $(x_n)_{n \in\bbbn} \subset \Omega^{m-1}(N,\F_N)$ be a sequence such that $d_Nx_n$ converges to $x\in \Omega_{(2)}^m(N, \F_N)$ for $n \to \infty$. Then $d_MT_v(p)x_n$ converges to $T_v(p)x$. The assumption implies that $T_v(p)x=d_My$ for some $y \in \dom d_M$. Now $$d_Nx_n=(1+ T_v(p)^{\aj}T_v(p))d_Nx_n-T_v(p)^{\aj}T_v(p)d_Nx_n=d_N Yd_Nx_n-T_v(p)^{\aj}d_MT_v(p)x_n \ .$$ In the limit one gets $x=d_N(Yx- T_v(p)^{\aj}y)$.
\end{proof}

Let $V$ be the closure of $d\Omega^{m-1}(N \cup M,\F)\oplus d^*\Omega^m(N \cup M,\F)$ in $\Omega^*_{(2)}(N \cup M,\F)$ and let $W$ be the orthogonal complement. The operators $\D$, $d$, $\tau$ map $V$, $W$ to themselves. Let $V_M$, $W_M$ be the intersection of $V$, $W$, respectively, with $\Omega_{(2)}^*(M,\F_M)$. We define an involution $\inv$ on $W$ by setting it $1$ on forms of degree smaller than $m$ and $-1$ on the orthogonal complement. We extend it by zero to a selfadjoint operator on $V \oplus W$. The operator $\D$ anticommutes with $\inv$. Let $\inv_M$ be the restriction of $\inv$ to $\Omega_{(2)}^*(M,\F_M)$. The operator $\D_M+ \iu\inv_M \tau_M$ is invertible and $\eta(M,g_M)$ is defined as the $\eta(\D_M,\iu\inv_M\tau_M)$. Analogously $\eta(N,g_N)$ is defined. 

\begin{prop} 
\label{prophigh}
Under the above assumption it holds that $$\eta(f,g_M,g_N)=\eta(N,g_N)-\eta(M,g_M) \ .$$
\end{prop}

\begin{proof}
We will use the following fact: If $A$ is a selfadjoint bounded operator on $\Omega_{(2)}^*(N\cup M,\F)$ vanishing on $V$ and anticommuting with $\inv$ and if $\D + A$ is invertible, then $\eta(\D,A)=\eta(N,g_N)-\eta(M,g_M)$. This follows from the calculation on the cylinder in the proof of \cite[Lemma 5.10]{waprod} and the Atiyah--Patodi--Singer index theorem for Dirac operators over $C^*$-algebras. For symmetric spectral sections, which are slightly less general, the statement has been proven in \cite{lpsign}. 

An operator $A$ with these properties is called a symmetric trivializing operator with respect to $\inv$.

The operator $\D|_V$ is invertible. Hence we may choose $\ve>0$ small enough such that $\phi_{4\ve}(\D)|_V=\phi_{4\ve}(\D|_V)=0$. Then $\Ran \phi_{4\ve}(\D) \subset W$. Since $A_{\alpha,1,\ve}(f)=\phi_{4\ve}(\D)A_{\alpha,1,\ve}(f)\phi_{4\ve}(\D)$, the operator $A_{\alpha,1,\ve}(f)$ decomposes as $0|_V \oplus A_{\alpha,1,\ve}(f)|_W$. Similarly the operators $\Li_{\alpha,1,\ve}$ and $\delta_{\alpha,\ve}$ decompose. Using this one checks that both commute with $\inv$. Furthermore, $\tau$ anticommutes with $\inv$ and the operators $(\tau \Li_{\alpha,1,\ve})^2$ and hence also $|\tau \Li_{\alpha,1,\ve}|$ commute with $\inv$. It follows that $A_{\alpha,1,\ve}(f)$ anticommutes with $\inv$. Thus it is a symmetric trivializing operator.
\end{proof}

\subsection{The $L^2$-case}
\label{L2}

Assume that $\A$ is endowed with a positive normal trace $\nu$ and denote by $\A_{\nu}$ its Hilbert space completion with respect to the scalar product $\langle a,b \rangle_{\nu}=\nu(a^*b)$. We get a Hilbert space $H=\Omega_{(2)}^*(N \cup M,\F)\ten_{\A} \A_{\nu}$. We denote by $\N$ the von Neumann algebraic closure of the $C^*$-algebra $B(\Omega_{(2)}^*(N \cup M,\F))$ in $B(H)$. The algebra $\N$ is endowed with an induced semifinite trace $\Tr_{\nu}$. It restricts to a semifinite trace on the von Neumann subalgebras $\N_M$ and $\N_N$ generated by $B(\Omega_{(2)}^*(M,\F_M))$ and $B(\Omega_{(2)}^*(N,\F_N))$, respectively. The operators $\D_M$ and $\D_N$ are affiliated to $\N_M$ and $\N_N$ with resolvents in $K(\N_M)$ and $K(\N_N)$, respectively. 

We write $$\eta_{\nu}(\D_M)=\frac{1}{\sqrt{\pi}}\int_0^{\infty} t^{-1/2}\Tr_{\nu}(\D_M e^{-t\D_M^2})~dt \ .$$ 
Analogously $\eta_{\nu}(\D_N)$ is defined.
For the well-definedness of this expression see for example \cite[\S 8]{ps}.

\begin{prop}
\label{L2prop}
It holds that 
$$\nu(\eta(f))=\eta_{\nu}(\D_N)-\eta_{\nu}(\D_M) \ .$$
\end{prop}

\begin{proof}
Let $(\alpha,1)\in \ins$ with $\alpha \neq 0$ and let $\ve \in (0,\infty)$.
From 
\begin{align*}
\langle\omega_1,|\tau\Li_{\alpha,1,\ve}|\hat\D_{\alpha,1,\ve}\omega_2 \rangle&=\langle\omega_1,\hat\D_{\alpha,1,\ve}\omega_2 \rangle_{\alpha,1,\ve} =\langle\hat\D_{\alpha,1,\ve}\omega_1,\omega_2 \rangle_{\alpha,1,\ve} \\
&=\langle \hat\D_{\alpha,1,\ve} \omega_1, |\tau\Li_{\alpha,1,\ve}|\omega_2\rangle =\langle |\tau\Li_{\alpha,1,\ve}|\hat \D_{\alpha,1,\ve} \omega_1, \omega_2\rangle
\end{align*}
we conclude that the invertible operator $|\tau\Li_{\alpha,1,\ve}| \hat\D_{\alpha,1,\ve}$ is selfadjoint. Up to our different conventions and an inaccuracy in \cite[\S 9]{ps} (where it is erronously assumed that the operator $\hat\D_{\alpha,1,\ve}|\tau\Li_{\alpha,1,\ve}|$ is selfadjoint) this operator is the one defined in \cite[Def. 9.13]{ps}. It holds that $|\tau\Li_{\alpha,1,\ve}| \hat\D_{\alpha,1,\ve}=\U_{\alpha,1,\ve}\tilde \D_{\alpha,1,\ve}\U_{\alpha,1,\ve}$. Thus the operator $A_{\alpha,\ve}^{PS}:=|\tau\Li_{\alpha,1,\ve}| \hat\D_{\alpha,1,\ve}-\D$ is $\ve$-spectrally concentrated. 

We define a continuous path of selfadjoint operators $(U_t)_{t\in [0,1+\alpha]}$ by setting $U_t=\U_{0,t,\ve}$ for $t \in [0,1]$ and $U_t=\U_{t-1,1,\ve}$ for $t \in [1,1+\alpha]$.
Then $U_0=1$ and $U_{1+\alpha}=\U_{\alpha,1,\ve}$; furthermore, $U_t$ is invertible and $U_t-1$ is $\ve$-spectrally concentrated for all $t$. It holds that $U_0 \tilde \D_{\alpha,1,\ve} U_0=\tilde \D_{\alpha,1,\ve}$ and $U_{1+\alpha} \tilde \D_{\alpha,1,\ve} U_{1+\alpha}=|\tau\Li_{\alpha,1,\ve}| \hat\D_{\alpha,1,\ve}=\D + A^{PS}_{\alpha,\ve}$.
Since $U_t \tilde \D_{\alpha,1,\ve} U_t$ is invertible for all $t \in [0,1+\alpha]$, its spectral flow vanishes. 

It follows that $$\eta(\D,A^{PS}_{\alpha,\ve})=\eta(\D,A(f)_{\alpha,1,\ve}) \ .$$
In particular $\eta(\D,A^{PS}_{\alpha,\ve})$ is independent of the choice of $\ve$ and depends only on the sign of $\alpha$.

One checks that the proof of the approximation result in \cite[Theorem 10.1]{ps} works also for our conventions and with the above inaccuracy corrected, proving that $$\nu(\eta(\D,A^{PS}_{\alpha,\ve})+\eta(\D,A^{PS}_{-\alpha,\ve}))=2 (\eta_{\nu}(\D_N)-\eta_{\nu}(\D_M)) \ .$$
The assertion now follows.
\end{proof}

Important examples are the following:

Let $\A=C^*\Gamma$ and $\pi:\bbbc \Gamma \to C^*\Gamma$ the inclusion.
\begin{itemize}
\item
Let $\nu_e$ be the finite trace on $C^*\Gamma$ defined by $\nu_e(g)=\delta_{g,e}$ for $g\in \Gamma$. Furthermore, let $\nu_1$ be the trace induced by the trivial representation $C^*\Gamma \to \bbbc$. The $L^2$-$\rho$-invariant of Cheeger and Gromov is defined as 
$$\rho_{(2)}(M):=\eta_{\nu_e}(\D_M)-\eta_{\nu_1}(\D_M) \ .$$
\item
Let $h_1,h_2: \Gamma \to U(k)$ be group homomorphisms. They induce homomorphisms $C^*\Gamma\to M_k(\bbbc)$. By pulling back the standard trace on $M_k(\bbbc)$ one obtains traces $\nu_{h_i}$ on $C^*\Gamma$. The induced Atiyah--Patodi--Singer $\rho$-invariant is defined as 
$$\rho_{h_1,h_2}(M):=\eta_{\nu_{h_1}}(\D_M)-\eta_{\nu_{h_2}}(\D_M) \ .$$
\end{itemize}

\section{Product formula}
\label{prodform}

Now let $M$, $N$ be Cartesian products with one common factor, i.~e. $M=M_1 \times M_2$ and $N=N_1 \times M_2$ and let $f_1:M_1 \to N_1$ be a smooth homotopy equivalence. We will only consider the case where $M_1$, $N_1$ are odd-dimensional and $M_2$ is even-dimensional. Other cases may be treated via suspension. If not specified, the definitions are as before.

Let $\tilde N_1 \to N_1,~\tilde M_2 \to M_2$ be Galois coverings with deck transformation group $\Gamma_i,~i=1,2$, respectively. Let $\B, \Ca$ be unital $C^*$-algebras endowed with a unital involutive algebra homomorphism $\bbbc\Gamma_1 \to \B$ and $\bbbc \Gamma_2 \to \Ca$, respectively. We set $\A=\B \ten \Ca$. Here we use the minimal tensor product. We identify $\B$ and $\Ca$ with the subalgebras $\B\ten \bbbc$ and $\bbbc \ten\Ca$ in $\A$, respectively.

We define the bundles $\F_{N_1}=\tilde N_1 \times_{\Gamma_1} \B$ on $N_1$ and $\F_2=\tilde M_2 \times_{\Gamma_2} \Ca$ on $M_2$ and $\F_N=\F_{N_1} \boxtimes \F_2$ on $N$ and endow them with the canonical $C^*$-scalar products and hermitian connections. This induces an $\A$-valued scalar product and hermitian connection on $\F_N$.

Let $\A_i$ be as in \S \ref{rhodef} and endow $\B_i:=\A_i \cap \B$ and $\Ca_i:=\A_i \cap \Ca$ with the subspace topology of $\A_i$. Note that $\B_i$ and $\Ca_i$ are closed under holomorphic functional calculus in $\B$ and $\Ca$, respectively. We assume that the projective limits $\B_{\infty}$ and $\Ca_{\infty}$ are dense in $\B_i$ and $\Ca_i$ for all $i$, respectively.  Let $\ideal_i$ be the closed ideal in $\Oi\A_i$ generated by the supercommutators $[\alpha,\beta]_s$ with $\alpha \in  \Oi\B_i$ and $\beta \in \Oi\Ca_i$.  In the following we deal with $\Oi^{\ideal_i}\A_i$.

We write $\di_1$ for the differential on $\Oi\B_i$ and $\di_2$ for the differential on $\Oi\Ca_i$.

For the bundles $\F_{N_1}$ and $\F_2$ there are projections $P_{N_1} \in \C(N_1,M_p(\bbbc\Gamma_1))$ and $P_2 \in \C(M_2,M_q(\bbbc\Gamma_2))$ as in \S \ref{rhodef}. We set $P_N=P_{N_1}\ten P_2$ and obtain on $\F^i_N$ the connection $P_N \di P_N$ in direction of $\A_i$.

As in \S \ref{setting} let $p_1:I^k \times M_1 \to N_1$ be a submersion associated to the homotopy equivalence $f_1: M_1 \to N_1$. (In this section the space $B$ will be a point.)

We consider the homotopy equivalence $f:=f_1 \times \id$ and the associated submersion $p:=p_1 \times \id$. 

Let $\F_{M_1}:=f_1^*\F_{N_1}$. Then $\F_M=f^*\F_N=\F_{M_1} \boxtimes \F_2$. We write $\F_1=\F_{N_1} \cup \F_{M_1}$ and $\F=\F_N \cup \F_M=\F_1 \boxtimes \F_2$. We define the projections $P_1=P_{N_1} \cup (P_{N_1} \circ f)$ and $P=P_1 \ten P_2$. Thus we have on $\F^i$ the connection $P \di P$ in direction of $\A_i$.

In the following we consider the spaces $\Omega^*_{(2)}(N,\F_N)$, $\Omega^*_{(2)}(M,\F_M)$ and $\Omega^*_{(2)}(N \cup M,\F)$ ungraded. For an operator $A$ on $\Omega_{(2)}^*(M_2,\F_2)$ we also denote by $A$ the operator $1 \ten A$ on $\Omega_{(2)}^*(M \cup N,\F)$. (Clearly, the tensor product is ungraded.) With this convention, the chirality operator on $\Omega_{(2)}^*(M \cup N,\F)$ decomposes as $\tau=\tau_1 \tau_2$ where $\tau_1,\tau_2$ are the chirality operators on $\Omega_{(2)}^*(M_1 \cup N_1,\F_1)$ and $\Omega_{(2)}^*(M_2,\F_2)$, respectively. The involution on $\Omega^*_{(2)}(N \cup M,\F)$ which equals the identity on forms of even degree and minus the identity on forms of odd degree  decomposes as $\gamma=\gamma_1\gamma_2$, where $\gamma_1,\gamma_2$ are the corresponding involutions on $\Omega_{(2)}^*(M_1 \cup N_1,\F_1)$ and $\Omega_{(2)}^*(M_2,\F_2)$, respectively.    

Only for the connection on $\Omega^*(N\cup M,\F)$ our convention is slightly different in order to preserve the usual formulas: Here we write $d=d_1 + d_2$, where $d_1(\alpha \wedge \beta)=(d_1\alpha) \wedge \beta$ and $d_2(\alpha \wedge \beta)=(-1)^k\alpha \wedge d_2\beta$ for $\alpha \in \Omega^k(N_1 \cup M_1,\F_1)$, $\beta \in \Omega^*(M_2,\F_2)$. 

The signature operator on $\Omega^*_{(2)}(M_1,\F_{M_1})$ is denoted by $\D_{M_1}$, and the signature operator on $\Omega^*_{(2)}(M_2,\F_2)$ by $\D_2$. On $\Omega^*(N_1 \cup M_1,\F_1)$ we define the operator $\D_1:=\D_{N_1} \oplus (-\D_{M_1})$. Furthermore, we have the operator $\D=\D_N \oplus (-\D_M)$ on $\Omega^*(N\cup M,\F)$, see \S \ref{setting}.

The operators $\D_1$, $\D_2$ and $\D$ are related via the equation
\begin{align*}
\D&=\tau_2\D_1 + \tau \gamma_1 \D_2\\
&=\frac 12(1+\tau_2)\D_1 - \frac 12 (1-\tau_2)\gamma_1\tau_1\D_1\tau_1\gamma_1 + \tau\gamma_1 \D_2 \ .
\end{align*}

Note that, as expected, $\D^2=\D_1^2+\D_2^2$.

The expression in the last line of the equation allows to associate an invertible bounded perturbation of $\D$ to an invertible bounded perturbation of $\D_1$:

If $A$ is a bounded selfadjoint operator on $\Hi_1:=\Omega_{(2)}^*(N_1 \cup M_1,\F_1)$, we define on $\Hi=\Omega_{(2)}^*(N\cup M,\F)$ the operator
$$\ov A:=\frac 12(1+\tau_2) A-\frac 12(1-\tau_2)\gamma_1\tau_1  A \tau_1\gamma_1 \ .$$

Thus
\begin{align}
\label{boundprod}
\D+ \ov{A}&=\frac 12(1+\tau_2)(\D_1 + A) - \frac 12 (1-\tau_2)\gamma_1\tau_1(\D_1+A)\tau_1\gamma_1 + \tau\gamma_1 \D_2 \ .
\end{align}

Now assume that $\D_1 + A$ is invertible on $\Hi_1$. Then the operator $\tau_2 \D_1 + \ov A$ is invertible on $\Hi$. Furthermore, this operator anticommutes with $\tau \gamma_1 \D_2$. It follows that $(\D + \ov{A})^2=(\tau_2\D_1+ \ov{A})^2+\D_2^2$ is invertible. Thus $\D+ \ov{A}$ is invertible.

From \S \ref{setting} we get an invertible perturbation $\D_1 + A(f_1)_{\alpha,1,\ve}$. 

\begin{prop}
\label{prodeta}
In $\Oi^{\ideal}\Ai$ modulo the closure of $[\Oi^{\ideal}\Ai,\Oi^{\ideal}\Ai]_s + (\di \Oi\Bi) \Oi\Ca_{\infty}+\Oi\Bi \di \Oi\Ca_{\infty}$ it holds that
$$\eta(\D,\ov{A(f_1)}_{\alpha,1,\ve})=\eta(\D_1,A(f_1)_{\alpha,1,\ve}) \ch(\ind (\D_2)) \ .$$
\end{prop}

\begin{proof}
For the proof we have to go back to the construction of the noncommutative $\eta$-form via superconnections, see \cite[\S 5]{wazyl}. We assume that $A$ is an adapted approximation of $A(f_1)_{\alpha,1,\ve}$ (see \cite[Def. 9.5]{wazyl} for this notion). Then $\ov{A}$ is adapted to $\D$ and (by definition) $\eta(\D,\ov{A(f_1)}_{\alpha,1,\ve})=\eta(\D,\ov A)$.

Let $\chi: \bbbr \to [0,1]$ be a smooth function such  that $\chi(x)=1$ for $x \ge 1$ and $\chi(x)=0$ for $x \le \frac 12$ and let $\sigma$ be the odd involution that generates the Clifford algebra $C_1$. In the following, it is considered as a formal parameter anticommuting with $\di_1$, $\di_2$.

Define the rescaled superconnections $$\spcb^{\B}_t:=P_1\di_1 P_1 + \sqrt t \sigma \tau_2(\D_1+ \chi(t)A) \ ,$$  
and $$\spcb^{\Ca}_t:=P_2\di_2 P_2 + \sqrt t \sigma \D_2 \ .$$ 

Note that these two superconnections anticommute with each other.

Furthermore, we set $\spc_t=P\di P + \sqrt t \sigma (\D+ \chi(t)\ov{A})$.

By definition $$\eta(\D_1,A)=\frac{1}{\sqrt{\pi}}\int_0^{\infty} \Tr_{\sigma}\sigma \frac{d\spc_t}{dt} e^{-\spc_t^2} ~dt \ .$$

Then 
$$\spc_t= \frac 12(1+\tau_2)\spcb^{\B}_t + \frac 12 (1-\tau_2)\gamma_1\tau_1 \spcb^{\B}_t\tau_1\gamma_1 + \tau\gamma_1 \spcb^{\Ca}_t$$
 and 
 $$\spc^2_t=\frac 12(1+\tau_2)(\spcb^{\B}_t)^2 +  \frac 12 (1-\tau_2)\gamma_1\tau_1 (\spcb^{\B}_t)^2\tau_1\gamma_1+ (\spcb^{\Ca}_t)^2 \ .$$ 
It follows that
\begin{align*}
\lefteqn{\sigma \frac{d\spc_t}{dt} e^{-\spc_t^2}}\\
&=\sigma \frac 12(1+\tau_2)\frac{d \spcb^{\B}_t}{dt}e^{-(\spcb^{\B}_t)^2} e^{-(\spcb^{\Ca}_t)^2} \\
& \quad + \sigma\frac 12 (1-\tau_2)\gamma_1\tau_1\frac{d \spcb^{\B}_t}{dt}e^{-(\spcb^{\B}_t)^2} e^{-(\spcb^{\Ca}_t)^2}\tau_1 \gamma_1\\
&\quad + \sigma \frac{d\spcb^{\Ca}_t}{dt} \bigl(\frac  12(1+\tau_2) e^{-(\spcb^{\B}_t)^2} e^{-(\spcb^{\Ca}_t)^2} + \frac 12 (1-\tau_2)\gamma_1\tau_1 e^{-(\spcb^{\B}_t)^2} e^{-(\spcb^{\Ca}_t)^2}\tau_1 \gamma_1\bigr)\ .
\end{align*}

Since $\frac{d\spcb^{\Ca}_t}{dt}$ anticommutes with $\gamma_2$ while the last factor in the last line commutes with $\gamma_2$, the trace $\Tr_{\sigma}$ of the last line vanishes.

Thus 
$$\Tr_{\sigma}\sigma \frac{d\spc_t}{dt} e^{-\spc_t^2}=\Tr_{\sigma}\sigma\frac{d \spcb^{\B}_t}{dt}e^{-(\spcb^{\B}_t)^2} e^{-(\spcb^{\Ca}_t)^2} \ .$$ 

Now the argument is as in the proof of \cite[Prop. 10.1]{wazyl}.

\end{proof}

The proposition implies that we get a product formula for $\eta(\D,A(f)_{\alpha,1,\ve})$ by comparing it with $\eta(\D,\ov{A(f_1)}_{\alpha,1,\ve})$. For that aim we consider how the other operators defined in \S \ref{setting} behave under taking products.

First we consider the property of being $\ve$-spectrally concentrated.

\begin{lem}
\begin{enumerate}
\item The operator $\phi_{\ve_1}(\D)\phi_{\ve_2}(\D_1)\phi_{\ve_2}(\D_2)$ equals
$\phi_{\ve_1}(\D)$ if $\ve_2>2\ve_1$ and it equals $\phi_{\ve_2}(\D_1)\phi_{\ve_2}(\D_2)$ if $\ve_1>4\ve_2$.
\item Assume that $A$ on $\Hi_1$ is $\ve_1$-spectrally concentrated with respect to $\D_1$ and $B$ on $\Hi_2$ is $\ve_2$-spectrally concentrated with respect to $\D_2$. Then $A\ten B$ is $\ve$-spectrally concentrated with respect to $\D$ for $\ve =2\max(\ve_1,\ve_2)$.
\end{enumerate}
\end{lem}

\begin{proof}
(1) From the Gelfand--Naimark theorem it follows that
 there is a $C^*$-homomorphism 
$$C_0([0,\infty)^2) \to B(\Hi), f\ten g \mapsto (f\ten g)(\D_1^2,\D_2^2):=f(\D_1^2)g(\D_2^2) \ .$$

For a positive function $f\in C_c([0,\infty))$ we can define the function $\tilde f(x,y):=f(x+y) \in C_0([0,\infty)^2)$. It holds that $\tilde f(\D_1^2,\D_2^2)=f(\D_1^2+\D_2^2)$.

Set $\psi_{\ve}(x)=\phi_{\ve}(\sqrt x)$ for $x\ge 0$. Then $\psi_{\ve}(\D^2)=\phi_{\ve}(\D)$. Note that $\psi_{\ve}(x)=0$ for $x \ge (\frac{\ve}{2})^2$ and $\psi_{\ve}(x)=1$ for $x \in [0,(\frac{\ve}{4})^2]$. 
Thus for $(\frac{\ve_2}{4})^2>(\frac{\ve_1}{2})^2$
$$\psi_{\ve_1}(x+y)\psi_{\ve_2}(x)\psi_{\ve_2}(y)=\psi_{\ve_1}(x+y) \ ,$$ 
and for $(\frac{\ve_1}{4})^2>2(\frac{\ve_2}{2})^2$
$$\psi_{\ve_1}(x+y)\psi_{\ve_2}(x)\psi_{\ve_2}(y)=\psi_{\ve_2}(x)\psi_{\ve_2}(y) \ .$$
The assertion now follows.

(2) The arguments are similar.
\end{proof}

We have that $T(p)=T(p_1) \ten 1$.  (We suppress the dependence on $v \in \Omega_c^*(I^k)$ from the notation.)  Furthermore, if $Y_1$ fulfills $1+T(p_1)^{\aj}T(p_1)=d_{N_1}Y_1+ Y_1d_{N_1}$, then $Y=Y_1\ten 1$ fulfills $1+T(p)^{\aj}T(p)=d_{N}Y+ Yd_{N}$. Clearly, the operators $\T_{\ve}(p)$ and $\Y_{\ve}$ from Lemma \ref{lemYeps} cannot be decomposed as easily. Therefore we need a variety of new operators which serve to interpolate between $\T_{\ve}(p)$ and $\Tl_{\ve}(p):=\T_{\ve}(p_1) \ten 1$. 

We set
\begin{align*}
\Tlp_{\ve_1,\ve_2}(p)&:=\phi_{\ve_2}(\D_2)\Tl_{\ve_1}(p)=\phi_{\ve_2}(\D_2)\phi_{\ve_1}(\D_{M_1})T(p_1)\phi_{\ve_1}(\D_{N_1})\ ,\\
\T^{ip}_{\ve_1,\ve_2}(p)&:=\phi_{\ve_1}(\D_M)\Tlp_{\ve_2,\ve_2}(p)\phi_{\ve_1}(\D_N) \ .
\end{align*}
Both are operators from $\Omega^*_{(2)}(N,\F_N)$ to $\Omega^*_{(2)}(M,\F_M)$.
The index $s$ stands for ``smoothing'' and the index $ip$ for ``interpolation''. Namely, for $\ve_1>4\ve_2$ by the previous lemma $\T^{ip}_{\ve_1,\ve_2}(p)=\Tlp_{\ve_2,\ve_2}(p)$, while for $\ve_2>2\ve_1$ it holds that $\T^{ip}_{\ve_1,\ve_2}(p)=\T_{\ve_1}(p)$.  

By the previous lemma the operator $\Tlp_{\ve_1,\ve_2}(p)$ is $2\max(\ve_1,\ve_2)$-spectrally concentrated. Furthermore, $\Tlp_{\ve_1,\infty}(p)=\Tl_{\ve_1}(p)$.

Let $J$ be a compact subset of $(0,\infty]\times (0,\infty) \cup (0,\infty) \times (0,\infty]$. Then $\T^{ip}(p)$ and $\Tlp(p)$, considered as families indexed by $(\ve_1,\ve_2) \in J$, define bounded operators between the Hilbert $C(J)$-modules $C(J,\Omega^*_{(2)}(N,\F_N))$ and $C(J,\Omega^*_{(2)}(M,\F_M))$ and between $C(J,\Omega^*_{H^1}(N,\F_N))$ and $C(J,\Omega^*_{H^1}(M,\F_M))$. 

We choose an operator $\Y^1$ on $C(J,\Omega^*_{(2)}(N_1,\F_{N_1}))$ such that $\T(p_1)$, $\Y^1$ fulfill an analogue of Lemma \ref{lemYeps}. In particular, there is a decomposition $\Y^1=(\Y^1)' + (\Y^1)''$.
 We set $\Yl=\Y^1 \ten 1$. Then the operators $\Tl(p)$, $\Yl$ fulfill an analogue of Lemma \ref{lemY}. Here we  understand $\Tl(p)$ as an operator from $C(J,\Omega^*_{(2)}(N,\F_N))$ to $C(J,\Omega^*_{(2)}(M,\F_M))$ and $\Yl$ as an operator on $C(J,\Omega^*_{(2)}(N,\F_N))$.

\begin{lem}
There are bounded operators $\Ylp$ and $\Y^{ip}$ on $C(J,\Omega^*_{(2)}(N,\F_N))$ such that an analogue of Lemma \ref{lemY} holds for $\Ylp$, $\Tlp(p)$ and for $\Y^{ip}$, $\T^{ip}(p)$. 
Furthermore, the operators $\Ylp$ and $\Y^{ip}$ can be chosen such that the following holds:

\begin{enumerate}
\item The operator $\Ylp$ decomposes as $(\Ylp)' + (\Ylp)''$ such that $(\Ylp)'_{\ve_1,\ve_2}$ is $2\max(\ve_1,\ve_2)$-spectrally concentrated and $(\Ylp)''$ commutes with $\D_N^2$. It holds that $\Ylp_{\ve_1,\infty}=\Yl_{\ve_1}$.
\item The operator $\Y^{ip}$ decomposes as $(\Y^{ip})' + (\Y^{ip})''$ such that $(\Y^{ip})'_{\ve_1,\ve_2}$ is spectrally concentrated if $\ve_1$ or $\ve_2$ is finite, and $(\Y^{ip})''$ commutes with $\D_N^2$. If $\ve_1>10\ve_2$, then $\Y^{ip}_{\ve_1,\ve_2}=\Ylp_{\ve_2,\ve_2}$. If $\ve_2>4\ve_1$, then $\Y^{ip}_{\ve_1,\ve_2}=\Y_{\ve_1}$. 
\end{enumerate}
\end{lem}

\begin{proof} We begin with $\Tlp(p)$.

We define $$(\Ylp)'_{\ve_1,\ve_2}=\phi_{\ve_2}(\D_2)^2((\Y^1)'_{\ve_1}\ten 1)  \ .$$
This operator is $2\max(\ve_1,\ve_2)$-spectrally concentrated.

With $h_{\ve_2}(x)=(1-\phi_{\ve_2}(x)^2)/x$ we set $$(\Ylp)''_{\ve_1,\ve_2}=h_{\ve_2}(\D_2)+\phi_{\ve_2}(\D_2)^2((\Y^1)''_{\ve_1}\ten 1) \ .$$

Using that 
$$\phi_{\ve_2}(\D_2)^2 +\Tlp_{\ve_1,\ve_2}(p)^{\aj}\Tlp_{\ve_1,\ve_2}(p)=d_M \phi_{\ve_2}(\D_2)^2((\Y^1)_{\ve_1}\ten 1) + \phi_{\ve_2}(\D_2)^2((\Y^1)_{\ve_1}\ten 1)d_N $$ 
and 
$$1-\phi_{\ve_2}(\D_2)^2=d_M h_{\ve_2}(\D_2) + h_{\ve_2}(\D_2) d_N \ ,$$ 
one checks that
 $\Ylp=(\Ylp)' + (\Ylp)''$ fulfills the equation in (1).

Since $\phi_{\ve_2}(\D_2)$ converges strongly to $1$ for $\ve_2 \to \infty$, it holds that $(\Ylp)_{\ve_1,\infty}=(\Y^1)_{\ve_1}\ten 1$.

Now we consider $\T^{ip}(p)$. We use ideas from the proof of Lemma \ref{lemYeps}. 
Let $\phi_{t,\ve}$ and $g_{\ve}$ be as defined there. 

We set
\begin{align*}
Z_{\ve_1,\ve_2}&=\int_0^1 g_{\ve_1}(d_M + d_M^*) \Tlp_{\ve_2,\ve_2}(p)\phi_{t,\ve_1}(d_N+ d_N^*) ~dt\\
&+\int_0^1  \phi_{t,\ve_1}(d_M + d_M^*)\Tlp_{\ve_2,\ve_2}(p)g_{\ve_1}(d_N+d_N^*) ~dt 
\end{align*}

and get
$$\Tlp_{\ve_2,\ve_2}(p)-\T^{ip}_{\ve_1,\ve_2}(p)=d_M Z_{\ve_1,\ve_2} + Z_{\ve_1,\ve_2} d_N \ .$$

The operator $Z_{\ve_1,\ve_2}$ is $2\ve_2$-spectrally concentrated. 

Choose continuous functions $\chi_1,\chi_2: J \to [0,1]$ with the following properties: It holds that $\chi_1(\ve_1,\ve_2)=0$ if $\ve_2< 3\ve_1$ and $\chi_1(\ve_1,\ve_2)=1$ if $\ve_2>4\ve_1$. 
Furthermore, $\chi_2(\ve_1,\ve_2)=0$ if $\ve_1< 9\ve_2$ and $\chi_2(\ve_1,\ve_2)=1$ if $\ve_1>10\ve_2$. 

Define
\begin{align*}
X_{\ve_1,\ve_2}&:=(1-\chi_1(\ve_1,\ve_2)-\chi_2(\ve_1,\ve_2))(\Ylp_{\ve_2,\ve_2} -Z_{\ve_1,\ve_2}^{\aj}\Tlp_{\ve_2,\ve_2}(p)+\T^{ip}_{\ve_1,\ve_2}(p)^{\aj}Z_{\ve_1,\ve_2}) \\
& \quad + \chi_1(\ve_1,\ve_2)\Y_{\ve_1}  + \chi_2(\ve_1,\ve_2)\Ylp_{\ve_2,\ve_2} \ .
\end{align*}
Since $-Z_{\ve_1,\ve_2}^{\aj}\Tlp_{\ve_2,\ve_2}(p)+\T^{ip}_{\ve_1,\ve_2}(p)^{\aj}Z_{\ve_1,\ve_2}$ is $2\ve_2$-spectrally concentrated, the operator 
$$(1-\chi_1(\ve_1,\ve_2)-\chi_2(\ve_1,\ve_2))\Bigl((\Ylp)'_{\ve_2,\ve_2} -Z_{\ve_1,\ve_2}^{\aj}\Tlp_{\ve_2,\ve_2}(p)+\T^{ip}_{\ve_1,\ve_2}(p)^{\aj}Z_{\ve_1,\ve_2}\Bigr)$$ is $2\ve_2$-spectrally concentrated.
We define $X'_{\ve_1,\ve_2}$ as the sum of this operator with $\chi_1(\ve_1,\ve_2)\Y_{\ve_1}'+\chi_2(\ve_1,\ve_2)(\Ylp)_{\ve_2,\ve_2}'$. 

Thus, if $\ve_2<\frac{1}{10} \ve_1$, then $X_{\ve_1,\ve_2}=\Ylp_{\ve_2,\ve_2}$ and $X_{\ve_1,\ve_2}'=(\Ylp_{\ve_2,\ve_2})$', while for $\ve_2>4\ve_1$ it holds that
$X_{\ve_1,\ve_2}=\Y_{\ve_1}$ and  $X_{\ve_1,\ve_2}'=\Y_{\ve_1}'$.

Now we set $\Y^{ip}=\frac 12 (X^{\aj} + X)$ and $(\Y^{ip})'=\frac 12((X')^{\aj}+X')$. 
\end{proof}

Using these definitions, we get operators $\tilde\D_{\alpha,\beta}^{\times}$, $\tilde\D_{\alpha,\beta}^{\times,s}$ and $\tilde\D_{\alpha,\beta}^{ip}$ on $C(J,\Hi)$ as in \S \ref{signop}, which are invertible for $(\alpha,\beta) \in \ins$ with $\alpha_0$ small enough and $|\alpha|\neq 0$.

From the previous lemma one obtains as in the proof of Lemma \ref{smooth}:

If $\ve_1,\ve_2 \in (0,\infty)$, then the operator $$\tilde \D_{\alpha,\beta,\ve_1,\ve_2}^{\times,s}-\D=:A^{\times,s}_{\alpha,\beta,\ve_1,\ve_2}$$ 
is smoothing since it is spectrally concentrated.

If $\ve_1 \in (0,\infty)$ or $\ve_2 \in (0,\infty)$, then the operator $$\tilde\D_{\alpha,\beta,\ve_1,\ve_2}^{ip}-\D=:A^{ip}_{\alpha,\beta,\ve_1,\ve_2}$$ 
is smoothing.  

If $\ve_2<\infty$, it holds that $$A^{ip}_{\alpha,\beta,\infty,\ve_2}=A^{\times,s}_{\alpha,\beta,\ve_2,\ve_2} \ ,$$
and if $\ve_1<\infty$, then $$A^{ip}_{\alpha,\beta,\ve_1,\infty}=A(f)_{\alpha,\beta,\ve_1}$$

For $\alpha \in \ins, ~|\alpha|\neq 0$ the $\eta$-form $\eta(\D,A^{ip}_{\alpha,1,\ve_1,\ve_2})$ is well-defined in $\Oi\Ai/\ov{[\Oi\Ai,\Oi\Ai]_s + \di\Oi\Ai}$ and does not depend on $(\ve_1,\ve_2) \in J$. Similarly,  $\eta(\D,A^{\times,s}_{\alpha,1,\ve_1,\ve_2})$ is well-defined and does not depend on $\ve_1,\ve_2 \in (0,\infty)$.
By choosing $J$ appropriately, we conclude:

\begin{prop}
\label{eqeta}
For $\ve, \ve_1,\ve_2 \in (0,\infty)$ and $|\alpha|\neq 0$ small enough
$$\eta(\D,A(f)_{\alpha,1,\ve})=\eta(\D,A^{\times,s}_{\alpha,1,\ve_1,\ve_2}) \ .$$
\end{prop}

Next we compare $\eta(\D,A^{\times,s}_{\alpha,1,\ve_1,\ve_2})$ with $\eta(\D,\ov{A(f_1)}_{\alpha,1,\ve})$.

For this we need some further technical considerations concerning the product structure of the involved operators.

In the following we fix $\ve_1,\ve_2 \in (0,\infty)$ and $(\alpha,\beta) \in \ins$ and omit the indices from the notation.

We write
$$\tilde \D_1=\D_1 +  A(f_1)=(-\iu)(\tilde \tau^1\tilde \delta^1 + \tilde \delta^1 \tilde \tau^1) \ .$$
Here $\tilde \tau^1$, $\tilde \delta^1$ are perturbations of $\tau_1$, $\iu d_1$ defined as in \S \ref{signop} using $\T(p_1)$, $\Y^1$ and $\gamma_1$.

Similarly, we define the following perturbations of $\tau$, $\iu d$: the operators $\tilde\tau^{\times}$, $\tilde\delta^{\times}$ using $\Tl(p)$, $\Yl$ and $\gamma$; the operators $\tilde\tau^{\times,s}$, $\tilde\delta^{\times,s}$ using  $\Tlp(p)$, $\Ylp$ and $\gamma$; and the operators $\tilde\tau^{ip}$, $\tilde\delta^{ip}$ using $\T^{ip}$, $\Y^{ip}$ and $\gamma$.

We write
\begin{align*}
\tilde\D^{\times}:&=(-\iu)(\tilde\tau^{\times}\tilde\delta^{\times} + \tilde\delta^{\times} \tilde\tau^{\times}) \ ,\\
\tilde\D^{\times,s}:&=(-\iu)(\tilde\tau^{\times,s}\tilde\delta^{\times,s} + \tilde\delta^{\times,s}\tilde\tau^{\times,s}) \ ,\\
\tilde \D^{ip}:&=(-\iu)(\tilde\tau^{ip}\tilde\delta^{ip} + \tilde\delta^{ip}\tilde\tau^{ip}) \ .
\end{align*}

Let $A$ be an operator on $\Hi_1:=\Omega_{(2)}^*(N_1,\F_{N_1}) \oplus \Omega_{(2)}^*(M_1,\F_{M_1})$ and choose $l_A \in \bbbz/2$. We decompose 
$$A=\gamma_1^{l_A} A^e + \gamma_1^{l_A+1} A^o \ .$$  
Here $A^e$ and $A^o$ are assumed to commute and anticommute with $\gamma_1$, respectively. 

The choice of $l_A$ uniquely determines $A^e$ and $A^o$ and vice versa by 
$$A^e=\frac 12 \gamma_1^{l_A}(A+\gamma_1 A \gamma_1)$$ 
and 
$$A^o=\frac 12 \gamma_1^{l_A+1}(A-\gamma_1 A \gamma_1) \ .$$

Recall that the operators $\tilde \tau^1, \tilde \delta^1$ are build from operators $\Li^1$, $\delta^1$ defined as in \S \ref{signop}. Similarly, the operators $\tilde \tau^{\times}, \tilde \delta^{\times}$ are build from operators $\Li^{\times}$, $\delta^{\times}$. 

These operators are closely related with each other:

We decompose 
\begin{align}
\label{decomposL}
\Li^1&=(\Li^1)^e + \gamma_1 (\Li^1)^o
\end{align} 
and
\begin{align}
\label{decomposd}
\delta^1&= \gamma_1(\delta^1)^e + (\delta^1)^o \ .
\end{align}
(These equations determine $l_A=1$ for $A=\delta^1$ and $l_A=0$ for $A=\Li^1$.)

Then it holds that $$\Li^{\times}=(\Li^1)^e + \gamma (\Li^1)^o$$ 
and
$$\delta^{\times}-\iu d_2= \gamma(\delta^1)^e + (\delta^1)^o \ .$$

In order to establish a similar relation between the operators $(-\iu)(\tilde \tau^1\tilde \delta^1 + \tilde \delta^1 \tilde \tau^1)$ and $(-\iu)(\tilde\tau^{\times}(\tilde\delta^{\times}-\iu d_2) + (\tilde\delta^{\times}-\iu d_2)\tilde\tau^{\times})$, we need the following lemma.

\begin{lem}
\label{changeinvol}
If $A$ is a bounded operator on $\Hi_1$ decomposed as $A=\gamma_1^{l_A}A^e + \gamma_1^{l_A+1}A^o$ with $l_A \in \bbbz/2$, then define the operator $A'=\gamma^{l_A}A^e + \gamma^{l_A+1}A^o$ on $\Hi$. 

\begin{enumerate}
\item If $l_{A^*}=l_A$, it holds that $(A')^*=(A^*)'$.
\item If $l_{AB}=l_A+l_B$, then 
$$A'B'=(AB)' \ .$$ 
\item Let $l_A=0$ and assume that $A$ is selfadjoint. For $f \in C_0(\bbbr)$  
$$f(A')=f(A)' \ .$$
\end{enumerate}
\end{lem}

\begin{proof}
(1) It holds that $(A^*)^e=(A^e)^*$ and $(A^*)^o=(-1)^{l_A+1}(A^o)^*$. Thus
$$(A^*)'=\gamma^{l_A}(A^e)^*+\gamma^{l_A+1}(-1)^{l_A+1}(A^o)^*=(\gamma^{l_A}A^e + \gamma^{l_A+1}A^o)^*=(A')^* \ .$$

(2) It holds that $(AB)^e=A^eB^e + (-1)^{l_B+1}A^oB^o$ and $(AB)^o=(-1)^{l_B}A^oB^e+A^eB^o$. Now it is straightforward to check that 
$$(AB)'=\gamma^{l_{AB}}(AB)^e +\gamma^{l_{AB}+1}(AB)^o=(\gamma^{l_A}A^e + \gamma^{l_A+1}A^o)(\gamma^{l_B}B^e + \gamma^{l_B+1}B^o)=A'B' \ .$$

(3) By (1) the operator $A'$ is selfadjoint. 
The map
$$C_0(\bbbr) \to B(\Hi),~ f\mapsto f(A')$$ 
is bounded. Since the maps $A \mapsto A^e$ and $A \mapsto A^o$ are continuous, the map
$$C_0(\bbbr) \to B(\Hi),~ f\mapsto f(A)'$$
is bounded as well. 

If the claim holds for $f,g \in C_0(\bbbr)$, then it holds as well for $f+g$ and, by (1),  for $fg$. 
It follows that we only have to prove the claim for a subset which generates a dense algebra in $C_0(\bbbr)$. We choose the functions $(x\pm \iu \lambda)^{-1}$ with $\lambda>0$ large enough such that $\|A^o\| \|(A^e \pm \iu \lambda)^{-1}\| <\frac 12$. Here we use that $A^e$ is also selfadjoint.

The Neumann series implies that
$$(A' \pm \iu \lambda)^{-1}= (A^e \pm \iu \lambda)^{-1} \sum_{j=0}^{\infty} (-1)^j (\gamma A^o (A^e \pm \iu \lambda)^{-1})^j \ .$$
In order to calculate $(A \pm \iu \lambda)^{-1})^e$ we  collect all terms in the Neumann series for $(A \pm \iu \lambda)^{-1}$ that commute with $\gamma_1$ and get 
\begin{align*}((A \pm \iu \lambda)^{-1})^e &=(A^e \pm \iu \lambda)^{-1} \sum_{j=0}^{\infty} (\gamma_1 A^o (A^e \pm \iu \lambda)^{-1})^{2j}\\
&=(A^e \pm \iu \lambda)^{-1} \sum_{j=0}^{\infty} (\gamma A^o (A^e \pm \iu \lambda)^{-1})^{2j}\\
&=((A' \pm \iu \lambda)^{-1})^e \ .
\end{align*}
Similarly, we collect all terms anticommuting with $\gamma_1$, multiply with $\gamma_1$ and get 
\begin{align*}
((A \pm \iu \lambda)^{-1})^o &=-\gamma_1(A^e \pm \iu \lambda)^{-1} \sum_{j=0}^{\infty} (\gamma_1 A^o (A^e \pm \iu \lambda)^{-1})^{2j+1}\\
&=-\gamma (A^e \pm \iu \lambda)^{-1} \sum_{j=0}^{\infty} (\gamma A^o (A^e \pm \iu \lambda)^{-1})^{2j+1}\\
&=((A' \pm \iu \lambda)^{-1})^o \ .
\end{align*}
Now the claim follows.
\end{proof}

We decompose
\begin{align}
\label{decompos}
\tilde \D_1=(-\iu)(\tilde \tau^1\tilde \delta^1 + \tilde \delta^1 \tilde \tau^1)&=E^e + \gamma_1 E^o \ .
\end{align}

The operator $\tau_2(\tilde \tau^{\times}(\tilde \delta^{\times}-\iu d_2) + (\tilde \delta^{\times}-\iu d_2)\tilde \tau^{\times})$ is constructed from $\Li^{\times}$ and $\delta^{\times}-\iu d_2$ using only the operations discussed in the previous Lemma. Thus we get from the decompositions \ref{decomposL} and \ref{decomposd} that
\begin{align}
\label{decompos2}
(-\iu)(\tilde \tau^{\times}(\tilde \delta^{\times}-\iu d_2) + (\tilde \delta^{\times}-\iu d_2)\tilde \tau^{\times})&=\tau_2(E^e+ \gamma E^o) \ .
\end{align}

\begin{prop}
\label{product}
Let $\ve_1,\ve_2 \in (0,\infty)$. For $\alpha_0>0$ small enough there is a path $(D_x)_{x\in [0,2]}$ of selfadjoint operators on $\Hi \oplus \Hi$ with $$D_0=(\D+A_{\alpha_0,1,\ve_1,\ve_2}^{\times,s})\oplus (\D + A_{-\alpha_0,1,\ve_1,\ve_2}^{\times,s})$$ 
and $$D_2=(\D+\ov{A(f_1)}_{\alpha_0,1,\ve_1})\oplus (\D+ \ov{A(f_1)}_{-\alpha_0,1,\ve_1})$$ 
such that $D_x-(\D \oplus \D)$ is bounded and depends on $x$ in a strongly continuous way and such that
$$\spfl((D_x)_{x\in [0,2]})=0 \ .$$
\end{prop}

\begin{proof}
Define the operator
$$V:=\frac 12(1+\tau_2) + \frac 12(1-\tau_2)\tilde \tau^{\times} \ .$$ 
Since $\tau_2$ commutes with $\tilde \tau^{\times}$, the operator $V$ is a unitary (even an involution) on $\Hi$. For $\alpha,\beta=0$ it equals $V_{0,0}=\frac 12(1+\tau_2) - \frac 12(1-\tau_2)\tau_1$.

Since $\tau_2$ commutes with $\tilde \delta^{\times}-\iu d_2$ as well, it holds that
$$V\bigl(\tilde \tau^{\times}(\tilde \delta^{\times}-\iu d_2) + (\tilde \delta^{\times}-\iu d_2)\tilde \tau^{\times}\bigr)V^{-1}=\tilde \tau^{\times}(\tilde \delta^{\times}-\iu d_2) + (\tilde \delta^{\times}-\iu d_2)\tilde \tau^{\times} \ .$$

Furthermore,
\begin{align*} 
V(\tilde \tau^{\times}d_2 + d_2\tilde \tau^{\times})V^{-1}
&=\frac 14(1+\tau_2)(\tilde \tau^{\times}d_2 + d_2\tilde \tau^{\times})(1+\tau_2)\\
&\quad +\frac 14(1+\tau_2)(d_2 + \tilde\tau^{\times}d_2\tilde \tau^{\times})(1-\tau_2)\\
&\quad +\frac 14(1-\tau_2)(d_2 + \tilde \tau^{\times}d_2\tilde \tau^{\times})(1+\tau_2)\\
&\quad + \frac 14(1-\tau_2)(\tilde \tau^{\times}d_2 + d_2\tilde \tau^{\times})(1-\tau_2) \ .
\end{align*}

We decompose $\tilde \tau_1=t^o + \gamma_1 t^e$. Since $\tau=\tau_1 \tau_2$, it holds that $\tilde \tau^{\times}=(t^o + \gamma t^e)\tau_2$. The differential $d_2$ anticommutes and $\tau_2$ commutes with $t^o$ and $\gamma t^e$, thus $\tilde \tau^{\times}d_2\tilde \tau^{\times}=-\tau_2 d_2\tau_2$. We conclude that the first and the last term vanish: $$\frac 14(1\pm \tau_2)(\tilde \tau^{\times}d_2 + d_2\tilde \tau^{\times})(1 \pm \tau_2)=\frac 14\tilde \tau^{\times} (1\pm \tau_2)(d_2 - \tau_2 d_2\tau_2)(1\pm \tau_2)=0 \ .$$
Evaluating the second and third term, we get
\begin{align*}
V(\tilde \tau^{\times}\tilde d_2^{\times} + \tilde d_2^{\times}\tilde \tau^{\times})V^{-1}&=d_2-\tau_2d_2 \tau_2=\gamma_1\D_2 \ .
\end{align*}

Thus
\begin{align}
\label{vdv}
V\tilde \D^{\times}V^{-1}&=(-\iu)(\tilde \tau^{\times}(\tilde \delta^{\times}-\iu d_2) + (\tilde \delta^{\times}-\iu d_2)\tilde \tau^{\times}) +\gamma_1\D_2 \ .
\end{align}

Eq. (\ref{decompos2}) implies that the difference between the operator $(-\iu)(\tilde \tau^{\times}(\tilde \delta^{\times}-\iu d_2) + (\tilde \delta^{\times}- \iu d_2)\tilde \tau^{\times})$ for arbitrary $\alpha,\beta$ and its value for $\alpha,\beta=0$ is bounded, since $E^{e/o}_{\alpha,\beta}-E^{e/o}_{0,0}$ is bounded.

Hence $V_{0,0}^{-1}V\tilde \D^{\times}V^{-1}V_{0,0}-\D$ is bounded.

Now choose real-valued functions $\alpha,\beta \in \C(\bbbr)$ that are decreasing for $x\le 0.5$ and increasing for $x \ge 0.5$ and such that
\begin{itemize}
\item
$\alpha(x)=\alpha_0$ for  $x \in (-\infty,0.1] \cup [0.9,\infty)$,
\item $\alpha(x)=0$ for $x \in [0.15,0.85]$, 
\item $\beta(x)=1$ for  $x \in (-\infty,0.2]\cup [0.8,\infty)$,
\item $\beta(x)=0$ for $x \in [0.3,0.7]$.
\end{itemize}

In particular $(\alpha(x),\beta(x))\in \ins$ for all $x\in\bbbr$. 

For $x \in [0,1]$ we define 
$$D^{pos}_x:=\left\{\begin{array}{ll} \tilde \D_{\alpha(x),\beta(x),\ve_1,\ve_2}^{\times, s}=\D + A_{\alpha(x),\beta(x),\ve_1,\ve_2}^{\times, s} & x \in [0,0.6)\\
V_{0,0}^{-1}V_{\alpha(x),\beta(x),\ve_1}\tilde \D_{\alpha(x),\beta(x),\ve_1}^{\times}V_{\alpha(x),\beta(x),\ve_1}^{-1} V_{0,0}  & x \in [0.6,1]  \ .
\end{array}\right. $$

Note that $D^{pos}_x- D^{pos}_0$ is bounded. It depends on $x$ in a strongly continuous way. At $x=0.6$ this holds since $A_{0,0,\ve_1,\ve_2}^{\times, s}=0$ and $\tilde \D_{0,0,\ve_1}^{\times}=\D$. 

Furthermore, $D^{pos}_0$ and $D^{pos}_1$ are invertible.

We show that the spectral flow of $(D^{pos}_x)_{x \in [0,1]}$ vanishes.

For that aim we define functions $\tilde \alpha,\tilde \beta:[0,2]\times \bbbr \to \bbbr$ by
\begin{align*}
\tilde \alpha(t,x)&=\left\{\begin{array}{ll} \alpha(x) & t\in [0,1],~ x \in \bbbr\\
(t-1)\alpha_0+(2-t)\alpha(x) & t = (1,2],~ x \in \bbbr \ ,
\end{array} \right. \\
\tilde \beta(t,x)&=\left\{\begin{array}{ll} t+(1-t)\beta(x) & t=[0,1),~ x \in \bbbr \\
1 & t\in [1,2],~ x \in  \bbbr \ .
\end{array} \right. 
\end{align*}

Furthermore, we define $\ve_2(x)=\psi(x)^{-1}\ve_2$, where $\psi \in \C(\bbbr)$ is a decreasing function with $\psi(x)=1$ for $x\le 0.4$ and $\psi(x)=0$ for $x \ge 0.6$. Note that $\ve_2(x)=\infty$ for $x \ge 0.6$.

We set for $t \in [0,2], x \in [0,0.6)$ 
$$D_{t,x}:= \tilde \D_{\tilde \alpha(t,x),\tilde \beta(t,x),\ve_1,\ve_2(x)}^{\times, s}$$ 
 and for $t \in [0,2], x \in [0.6,1]$ 
$$D_{t,x}:=V_{\tilde \alpha(t,0.6),\tilde \beta(t,0.6),\ve_1}^{-1}V_{\tilde \alpha(t,x),\tilde \beta(t,x),\ve_1}\tilde \D_{\tilde \alpha(t,x),\tilde \beta(t,x),\ve_1}^{\times}V_{\tilde \alpha(t,x),\tilde \beta(t,x),\ve_1}^{-1} V_{\tilde \alpha(t,0.6),\tilde \beta(t,0.6),\ve_1} \ .$$

The family $(D_{t,x})_{(t,x)\in [0,2]\times [0,1]}$ defines a regular selfadjoint operator on $C([0,2]\times [0,1],\Hi)$ with domain $C([0,2]\times [0,1],\Hi^1)$. Here we use that $\tilde \D_{ \alpha, \beta,\ve_1,\infty}^{\times, s}=\tilde \D_{\alpha,\beta,\ve_1}^{\times}$. The operator $D_{t,x}$ defines a regular selfadjoint operator on $C([0,2]\times [0,1],\Hi)$ with domain $C([0,2]\times [0,1],\Hi^1)$. Note that $D^{pos}_x=D_{0,x}$.

Since $D_{t,0}$ and $D_{t,1}$ are invertible for any $t \in [0,2]$, the homotopy invariance of the spectral flow \cite[\S 4]{wancsf} implies that 
$$\spfl((D_{0,x})_{x \in [0,1]})=\spfl((D_{2,x})_{x\in [0,1]}) \ .$$ 
The right hand side vanishes since $D_{2,x}$ is invertible for any $x \in [0,1]$.

We define $D^{neg}_x$ in the same way as $D^{pos}_x$ but we replace the function $\alpha(x)$ by $-\alpha(x)$.
As before, one proves that $\spfl((D^{neg}_x)_{x\in [0,1]})=0$.

For $x \in [0,1]$ set $D_x:=(D^{pos}_x \oplus D^{neg}_x)$.

Now we compare the operator $V_{0,0}^{-1}V\tilde \D^{\times}V^{-1}V_{0,0}$ with the operator $\D+\ov{A(f_1)}$. 

We wrote in (\ref{decompos})
$$\D_1 +  A(f_1)=E^e + \gamma_1 E^o \ ,$$
and thus by (\ref{boundprod})
$$\D+ \ov{A(f_1)}=\frac 12(1+\tau_2)(E^e + \gamma_1 E^o) -\frac 12(1-\tau_2)\gamma_1\tau_1 (E^e + \gamma_1 E^o)\tau_1\gamma_1+
\tau \gamma_1 \D_2 \ .$$
We calculate
\begin{align*}
V_{0,0}(\D+\ov{A(f_1)})V_{0,0}^{-1}
&=\frac 12(1+\tau_2)(E^e + \gamma_1 E^o)-\frac 12(1-\tau_2)\gamma_1(E^e + \gamma_1 E^o)\gamma_1 \\
&\quad - \frac 12(1+\tau_2)\tau\gamma_1 \D_2 \tau_1 - \frac 12(1-\tau_2)\tau_1 \tau\gamma_1 \D_2 \\
&=\tau_2 E^e + \gamma_1 E^o + \gamma_1 \D_2 \ .
\end{align*}

On the other hand, using (\ref{decompos2}) and (\ref{vdv})
we get that $$V\tilde\D^{\times}V^{-1}=\tau_2(E^e+ \gamma E^o)+\gamma_1\D_2 \ .$$

These equations imply that the involution $\tau_2 \gamma_2$ commutes with $V_{0,0}(\D+\ov{A(f_1)})V_{0,0}^{-1}$ and with $V\tilde \D^{\times}V^{-1}$. On the positive eigenspace of $\tau_2\gamma_2$ both operators agree with each other. 

When restricted to the negative eigenspace of $\tau_2\gamma_2$, the operator $V_{0,0}(\D+\ov{A(f_1)})V_{0,0}^{-1}$ agrees with $\tau_2(E^e - \gamma E^o) +\gamma_1\D_2$.

Note that $$E^e_{\alpha,\beta,\ve_1}=E^e_{-\alpha,\beta,\ve_1}$$ 
and
$$E^o_{\alpha,\beta,\ve_1}=-E^o_{-\alpha,\beta,\ve_1} \ ,$$ since replacing $\alpha$ by $-\alpha$ in $\tilde \delta^1$ is equivalent to replacing $\gamma_1$ by $-\gamma_1$. (The formal proof of these equations uses arguments as in the proof of Lemma \ref{changeinvol}.)

Thus on the negative eigenspace of $\tau_2\gamma_2$ the operator $V_{0,0}(\D+\ov{A(f_1)}_{\alpha,\beta,\ve_1})V_{0,0}^{-1}$ agrees with the operator $V_{-\alpha,\beta,\ve_1}\tilde \D_{-\alpha,\beta,\ve_1}^{\times}V_{-\alpha,\beta,\ve_1}^{-1}$.

Recall that 
$$D_1=(V_{0,0}^{-1}V_{\alpha_0,1,\ve_1}\tilde \D_{\alpha_0,1,\ve_1}^{\times}V_{\alpha_0,1,\ve_1}^{-1} V_{0,0}) \oplus (V_{0,0}^{-1}V_{-\alpha_0,1,\ve_1}\tilde \D_{-\alpha_0,1,\ve_1}^{\times}V_{-\alpha_0,1,\ve_1}^{-1} V_{0,0}) \ .$$
For $x \in [1,2]$ set
\begin{align*}
D_x:&=(2-x)D_1 + (x-1)\bigl((\D+\ov{A(f_1)}_{\alpha_0,1,\ve_1})\oplus (\D+\ov{A(f_1)}_{-\alpha_0,1,\ve_1})\bigr) \ .
\end{align*}

It follows from the previous considerations that the spectral flow of $(D_x)_{x\in [1,2]}$ vanishes. Furthermore, $D_x-(\D \oplus \D)$ is bounded and depends in a strongly continuous way on $x$.

Thus the path $(D_x)_{x \in [0,2]}$ fulfills the conditions of the proposition.
\end{proof}

\begin{cor}
Let $\ve_1,\ve_2 \in (0,\infty)$. For $\alpha>0$ small enough it holds that 
$$\eta(\D,A_{\alpha,1,\ve_1,\ve_2}^{\times,s}) +\eta(\D,A_{-\alpha,1,\ve_1,\ve_2}^{\times,s})=\eta(\D,\ov{A(f_1)}_{\alpha,1,\ve_1})+ \eta(\D,\ov{A(f_1)}_{-\alpha,1,\ve_1}) \ .$$
\end{cor}

From the corollary and Prop. \ref{prodeta} and \ref{eqeta} we conclude:

\begin{theorem}
\label{prodtheo}
In $\Oi^{\ideal}\Ai$ modulo the closure of $[\Oi^{\ideal}\Ai,\Oi^{\ideal}\Ai] + (\di_1\Oi\Bi) \hat \Omega_*\Ca_{\infty}+ \Oi\Bi \di_2 \hat \Omega_*\Ca_{\infty}$ it holds that
$$\eta(f)=\eta(f_1) \ch(\ind (\D_2)) \ .$$
\end{theorem}

We also prove a product formula for the higher $\eta$-invariants from \S\ref{high}. The proof is independent of the previous results. It was announced in the introduction of \cite{waprod} and is given here since it uses arguments from above. We use the notation from before, however the manifold $N$ does not play any role in the following.

Assume that the ranges of the closure of $d:\Omega^{m-1}(M,\F_M)\to \Omega_{(2)}^m(M, \F_M),~ m=(\dim M+1)/2$ and of $d_1:\Omega^{m_1-1}(M_1,\F_{M_1})\to \Omega_{(2)}^{m_1}(M_1, \F_{M_1}),~ m_1=(\dim M_1+1)/2$ are closed. Then the higher $\eta$-invariants $\eta(M_1,g_{M_1})$ and $\eta(M,g_M)$ are well-defined, see \S \ref{high}. 

\begin{prop}
\label{prodhigh}
In $\Oi^{\ideal}\Ai$ modulo the closure of $[\Oi^{\ideal}\Ai,\Oi^{\ideal}\Ai] + (\di_1\Oi\Bi) \hat \Omega_*\Ca_{\infty}+ \Oi\Bi \di_2 \hat \Omega_*\Ca_{\infty}$ it holds that
$$\eta(M,g_M)=\eta(M_1,g_{M_1}) \ch(\ind (\D_{M_2})) \ .$$
\end{prop}

\begin{proof}
Set $A=\iu \inv_{M_1} \tau_{M_1}$ and, as before, $$\ov A=\frac 12(1+\tau_{M_2}) A-\frac 12(1-\tau_{M_2})\gamma_{M_1}\tau_{M_1}  A\tau_{M_1}\gamma_{M_1} \ .$$
We have by a proof analogous to the proof of Prop. \ref{prodeta} that
$$\eta(\D_{M},\ov A)=\eta(\D_{M_1},A)\ch(\ind(\D_{M_2})) \ .$$

Since $\gamma_{M_1}$ and $\tau_{M_1}$ anticommute with $A$, we have that $\ov A=\tau_{M_2} A=\iu\inv_{M_1} \tau_{M_1} \tau_{M_2}$. Thus $\ov A$ anticommutes with $\inv_M$. It follows that $\ov A$ is a symmetric trivializing operator with respect to $\inv_M$.

Since $\eta(M,g_M)$ does not depend on the choice of the symmetric trivializing operator, it holds that $\eta(M,g_M)=\eta(\D_{M},\ov A)$.
\end{proof}

\section{Additivity}

Let $L$, $M$, $N$ be odd-dimensional closed oriented Riemannian manifolds and let $f:M \to N$ and $g:L \to M$ be oriented homotopy equivalences.

With $\F_M$, $\F_N$ as in \S \ref{setting} we define $\F_L=g^*\F_M$.

\begin{prop}
\label{propadd}
It holds that $$\eta(f \circ g)+\eta(\id_M)=\eta(f)+\eta(g) \ .$$
\end{prop}

\begin{proof}
Let $p:I^k \times M \to N$, $q:I^l \times L \to M$ be submersions and let $v \in \Omega^k_c(I^k)$, $w\in \Omega^l_c(I^l)$ as in \S \ref{setting}. We set $\tilde \Hi_1:=\Omega^*_{(2)}(N,\F_N)\oplus \Omega^*_{(2)}(M,\F_M)$, $\tilde\Hi_2:= \Omega^*_{(2)}(M^{op},\F_M) \oplus \Omega^*_{(2)}(L^{op},\F_L)$ and $\tilde\Hi:=\tilde\Hi_1 \oplus \tilde\Hi_2$. On $\tilde\Hi$ we have the signature operator 
 $$\D=\D_N\oplus \D_M \oplus (-\D_M) \oplus (-\D_L) \ .$$ 

From \S \ref{setting} we get operators $T_v(p) \oplus 1:\tilde \Hi_1 \to \Omega^*_{(2)}(M^{op},\F_M) \oplus \Omega^*_{(2)}(M,\F_M)$ and $1 \oplus T_w(q):\Omega^*_{(2)}(M^{op},\F_M) \oplus \Omega^*_{(2)}(M,\F_M) \to \tilde \Hi_2$. Define the path of unitaries $$U(t)=\left(\begin{array}{cc} e^{i\pi t}\cos(\frac \pi 2 t) & \sin(\frac \pi 2 t) \\ -e^{i\pi t}\sin(\frac \pi 2 t) & \cos(\frac \pi 2 t)\end{array}\right), t \in [0,1]$$ on $\Omega^*_{(2)}(M^{op},\F_M) \oplus \Omega^*_{(2)}(M,\F_M)$. Then $U(0)=1$, and $U(1)$ interchanges the two summands.

Set
$$T^{\oplus}(t)=(1 \oplus T_w(q))U(t)(T_v(p) \oplus 1):\tilde\Hi_1 \to \tilde\Hi_2 \ .$$
Note that 
$$T^{\oplus}(0)=T_v(p)\oplus T_w(q)$$ 
and 
$$T^{\oplus}(1)=\left(\begin{array}{cc} 0 & 1\\ T_w(q)T_v(p) & 0 \end{array}\right) \ .$$

Define $r=p \circ(\id \times q)):I^{k+l}\times L \to N$. Then, by the calculation in the proof of Lemma \ref{iso},
$$T_w(q)T_v(p)=T_{w\wedge v}(r) \ .$$

We choose $Y(p)$ and $Y(q)$ as in Lemma \S \ref{lemY} for $T_v(p)$ and $T_w(q)$, respectively.
 
Using that 
$$U(t)^{\aj}=\left(\begin{array}{cc} -\tau_M & 0\\ 0 & \tau_M \end{array}\right)U(t)^*\left(\begin{array}{cc} -\tau_M & 0\\ 0 & \tau_M \end{array}\right)=JU(t)^*J$$
with $$J=\left(\begin{array}{cc} -1 & 0\\ 0 & 1 \end{array}\right) \ ,$$
one gets 
\begin{align*}
\lefteqn{1+(T^{\oplus}(t))^{\aj}T^{\oplus}(t)}\\
&=1+(T_v(p)^{\aj} \oplus 1)JU(t)^*J(1 \oplus (d_M Y(q)+Y(q)d_M-1))U(t)(T_v(p) \oplus 1)\\
&=1+(d_N \oplus d_M)(T_v(p)^{\aj} \oplus 1)JU(t)^*J(0 \oplus Y(q))U(t)(T_v(p) \oplus 1)\\
&\quad +(T_v(p)^{\aj} \oplus 1)JU(t)^*J(0 \oplus Y(q))U(t)(T_v(p) \oplus 1)(d_N \oplus d_M)\\
&\quad +(T_v(p)^{\aj}T_v(p) \oplus (-1))\\
&=(d_N \oplus d_M)Y^{\oplus}(t) + Y^{\oplus}(t) (d_N \oplus d_M)
\end{align*}
with $$Y^{\oplus}(t)=(T_v(p)^{\aj} \oplus 1)U(t)^{\aj}(0 \oplus Y(q))U(t)(T_v(p) \oplus 1)+(Y(p)\oplus 0) \ .$$
Thus we have defined $Y^{\oplus}(t)$ such that $T^{\oplus}(t)$, $Y^{\oplus}(t)$ fulfill an analogue of Lemma \ref{lemY}. 

For $\ve\in (0,\infty)$ we set, with $\phi_{\ve}$ as before,
$$\T^{\oplus}(t)_{\ve}:=\phi_{\ve}(\D_M \oplus \D_L) T^{\oplus}(t) \phi_{\ve}(\D_N \oplus \D_M) \ .$$

We get an analogue of Lemma \ref{lemYeps} for the operator $\T^{\oplus}(t)_{\ve}$. Then, as in \S \ref{signop} one obtains selfadjoint bounded operators $A^{\oplus}(t)_{\alpha,\beta,\ve}$ on $\tilde \Hi$ such that $\D+A^{\oplus}(t)_{\alpha,1,\ve}$ is invertible for all $t \in [0,1]$ for $|\alpha|\neq 0$ small enough. By the definition in \S \ref{rhodef}
$$\frac 12\bigl(\eta(\D, A^{\oplus}(0)_{\alpha,1,\ve})+ \eta(\D, A^{\oplus}(0)_{-\alpha,1,\ve})\bigr)=\eta(f)+\eta(g)$$
and
$$\frac 12\bigl(\eta(\D, A^{\oplus}(1)_{\alpha,1,\ve})+ \eta(\D, A^{\oplus}(1)_{-\alpha,1,\ve})\bigr)=\eta(f\circ g)+\eta(\id_M) \ .$$
Furthermore, $\eta(\D, A^{\oplus}(t)_{\pm \alpha,1,\ve})$ does not depend on the parameter $t$. This shows the assertion.
\end{proof}

\section{$\rho$-forms and the structure set}
\label{structure}

In this section we show that the form $\rho(f)$ from \S \ref{rhodef} is well-defined on the surgery structure set. The main conceptional input in the proof is a generalization of the framework of \cite{hs} to manifolds with cylindrical ends.

Let $M$, $N$ be even-dimensional oriented Riemannian manifolds with cylindrical ends $Z_M^+$ and $Z_N^+$, respectively, and let $\tilde N \to N$ be a Galois covering with deck transformation group $\Gamma$. We denote the cross-sections of the cylindrical ends by $\ra M$ and $\ra N$ and fix an isometry of the cylindrical ends with $\bbbr^+ \times \ra M$ and $\bbbr^+ \times \ra N$, respectively. 

We define the cylinders $Z_M=\bbbr \times \ra M$ and $Z_N=\bbbr \times \ra N$ and denote the first coordinate by $x$.

Let $f:M \to N$ be a smooth map of degree one which is of the product form $\id_{\bbbr} \times f_{\ra}$ on the cylindrical end. (When carrying over the constructions from \S \ref{setting} this assumption is made tacitly whenever it applies.) 

We write $Z=Z_N \cup Z_M^{op}$ and $Z^+=Z_N^+ \cup (Z_M^+)^{op}$. 

The flat bundles $\F_N$ , $\F_M$, $\F$ are defined as in \S \ref{setting}. On the cylindrical ends they are pullbacks of bundles $\F_{\ra N}$, $\F_{\ra M}$, $\F_{\ra}$, respectively. We define $\F_Z$ as the pullpack of $\F_{\ra}$ to $Z$.

We carry over the constructions in \S \ref{setting} to the manifolds with cylindric ends $M$ and $N$. Thus we get $Q$, $\D$, $\tau$ and spaces $\Hi$, $\Hi^1$. If we apply the constructions to the boundary $\ra M$, $\ra N$, we add an index $\ra$. We have $Q_{\ra}$, $\D_{\ra}$, $\tau_{\ra}$ and the spaces $\Hi_{\ra}$, $\Hi_{\ra}^1$.
(We assume $B$ to be a point.)

We may choose a submersion $p:I^k\times M \to N$ as in \S \ref{setting} with the additional condition that $p|_{I^k \times Z_M^+}$ is of the product form $\id_{\bbbr} \times p_{\ra}$ for a submersion $p_{\ra}:I^k \times \ra M \to \ra N$. 

Then the operator $T_v(p)$ is defined as in \S \ref{setting}. We may choose the operator $Y$ from Lemma \ref{lemY} such that its restriction to the cylindrical end is of the form $\gamma_1 \ten Y^{\ra}$.

From \S \ref{signop} we get operators $\tilde \tau_{\alpha,\beta,\ve}^{\ra}, \tilde \delta_{\alpha,\beta,\ve}^{\ra}$, such that 
$$\tilde \D_{\alpha,\beta,\ve}^{\ra}=(-\iu)(\tilde \tau_{\alpha,\beta,\ve}^{\ra}\tilde \delta_{\alpha,\beta,\ve}^{\ra}+\tilde \delta_{\alpha,\beta,\ve}^{\ra}\tilde \tau_{\alpha,\beta,\ve}^{\ra})$$ 
is  an $\ve$-spectrally concentrated perturbation of the signature operator $\D_{\ra}$ on  $\Hi_{\ra}$ for $\ve \in (0,\infty)$ and $(\alpha,\beta) \in \ins$.
Here again we assume $\alpha_0>0$ small enough such that the following constructions work.

Now we consider the construction of \S \ref{signop} for the manifolds $M$, $N$.
If not specified we assume $\ve=\infty$ and omit it from the notation.

As in \S \ref{signop}, with $\T_v(p):=T_v(p)$ and $\Y:=Y$, we define the operators $\Li_{\alpha,\beta}$, $\delta_{\alpha}$, $\tilde \tau_{\alpha,\beta}$, $\tilde \delta_{\alpha,\beta}$ and
\begin{align}
\label{hsop}
\tilde \D_{\alpha,\beta}=\tilde \delta_{\alpha,\beta}-\tilde \tau_{\alpha,\beta} \tilde \delta_{\alpha,\beta}\tilde \tau_{\alpha,\beta} \ .
\end{align}
Recall that this operator anticommutes with $\tilde \tau_{\alpha,\beta}$. 

In the following, for simplicity we omit the suffixes $\alpha,\beta$.

On the cylindrical end it holds that $\tau=\iu\gamma_1\tau_1\tau_{\ra}$ with $\tau_1=\iu c(dx)$.

Using the definitions from before Lemma \ref{changeinvol}, we can decompose $\tilde \tau^{\ra}=t^{o}+ \gamma_{\ra}t^{e}$. 
Then by an analogue of Lemma \ref{changeinvol} on the cylindrical end 
\begin{align}
\label{eqtau}
\tilde \tau&=-\iu \tau_1(\gamma_1 t^{o}+\gamma t^e) \ .
\end{align} 
(Here we follow the same convention as in \S \ref{prodform} concerning tensor products; thus for example $t^o=1\ten t^o$ on $\Omega_{(2)}^*(\bbbr)\ten \Hi_{\ra}$ in an ungraded sense. This accounts for the extra factor $\gamma_1$ in the formulas here compared with those in Lemma \ref{changeinvol}.)

We  also decompose $\tilde \delta^{\ra}=d^{o}+ \gamma_{\ra} d^{e}$ and define $\tilde \delta^{cs}:=\tilde \delta-dx\, \ra_x$. Here the suffix $cs$ stands for ``cross section''. 
Then
\begin{align}
\label{eqcs}
\tilde \delta^{cs}=(-\iu)(\gamma_1 d^o + \gamma d^e) \ .
\end{align}

In particular, on the cylindrical end $\tilde \tau$ and $\tilde \delta$ are invariant under translation and extend to operators on $\Omega^*_{(2)}(Z,\F_Z)$, which we denote by $\tilde \tau^Z$ and $\tilde \delta^Z$, respectively. We define 
$$\tilde \D^Z=\tilde \delta^Z-\tilde \tau^Z \tilde \delta^Z \tilde \tau^Z \ .$$

For $s\in \Omega_c^*(Z^+,\F)$ it follows that $\tilde \D s\in \Omega_c^*(Z^+,\F)$  and $\tilde \D s=\tilde \D^Zs$.

Note that $\tilde \tau_Z$ and $\tilde \delta^{cs}$ are also well-defined on $\C(\ra N \cup \ra M , \Lambda^* T^*Z \ten\F_Z)$ because no $x$ differentiation is involved.

We go on to study the situation on the cylinder in more detail. 

Eq. (\ref{eqtau}) implies that $c(dx)\tilde \tau^Z=-\tilde \tau^Z c(dx)$ and that for $\omega \in \Omega^*(\ra N \cup \ra M, \F_{\ra})$
$$\tilde \tau^Z (dx \wedge \omega)=\iu\tau_1 dx \wedge \tilde \tau^{\ra}\omega=\tilde \tau^{\ra}\omega \ ,$$
 $$\tilde \tau^Z (\omega)=-\iu\tau_1\gamma_1 \tilde \tau^{\ra}(\omega)=dx \wedge \tilde \tau^{\ra}\omega \ .$$

It holds that
$$\tilde \D^Z=c(dx)\left(\ra_x - \tilde \tau^Z \iu c(dx)(\tilde \tau^Z \iu\tilde\delta^{cs}- \iu\delta^{cs}\tilde \tau^Z)\right) \ .$$

Let $$\tilde \Phi: \Omega^*(\ra N \cup \ra M, \F_{\ra}) \to \C(\ra N \cup \ra M , \Lambda^* T^*Z \ten\F_Z) \ ,$$
$$\omega \mapsto \frac{1}{\sqrt 2}\bigl(dx \wedge \omega + \tilde \tau^Z(dx \wedge \omega)\bigr)=\frac{1}{\sqrt 2}(dx \wedge \omega +\tilde \tau^{\ra}\omega) \ .$$

The operator $\tilde \Phi$ is an isomorphism onto the positive eigenspace of $\tilde \tau^Z$.

The operator $c(dx)\tilde \D^Z+\ra_x$ is even with respect to $\tilde \tau^Z$ and well-defined on $\C(\ra N \cup \ra M , \Lambda^* T^*Z \ten\F_Z)$. By definition, the boundary operator of $\tilde \D$ is the restriction of $c(dx)\tilde \D^Z+\ra_x$ to the positive eigenspace of $\tilde \tau^Z$ on $\C(\ra N \cup \ra M , \Lambda^* T^*Z \ten\F_Z)$. It equals $\iu c(dx)(\tilde \tau^Z \iu\tilde\delta^{cs}- \iu\tilde \delta^{cs}\tilde \tau^Z)$.

We evaluate the boundary operator further:
Using that for $\omega \in \Omega^*(\ra N \cup \ra M, \F_{\ra})$, by eq. (\ref{eqcs}),
$$\iu\tilde \delta^{cs}(dx \wedge \omega)=- dx \wedge \tilde \delta^{\ra}\omega$$
$$\iu\tilde \delta^{cs}\omega=\tilde \delta^{\ra}\omega \ ,$$
we get that
\begin{align*}
\lefteqn{\tilde \Phi^{-1}\iu c(dx)(\tilde \tau^Z \iu \tilde \delta^{cs} - \iu\tilde \delta^{cs} \tilde \tau^Z)\tilde \Phi(\omega)}\\
&=\frac{1}{\sqrt 2}\tilde \Phi^{-1}\iu c(dx)(\tilde \tau^Z\iu \tilde \delta^{cs} - \iu \tilde \delta^{cs}\tilde \tau^Z)\bigl(dx \wedge \omega + \tilde \tau^Z(dx \wedge \omega))\\
&=\frac{1}{\sqrt 2}\tilde \Phi^{-1}\iu c(dx)(-\tilde \tau^Z dx \wedge \tilde \delta^{\ra} \omega  + \tilde \tau^Z\tilde \delta^{\ra}\tilde \tau^{\ra}\omega  - \tilde \delta^{\ra}\tilde \tau^{\ra}\omega + dx \wedge \tilde \delta^{\ra} \omega )\\
&=\frac{1}{\sqrt 2}\tilde \Phi^{-1}\iu c(dx)(-\tilde \tau^{\ra}\tilde \delta^{\ra} \omega  + dx \wedge \tilde \tau^{\ra}\tilde \delta^{\ra}\tilde \tau^{\ra}\omega  - \tilde \delta^{\ra}\tilde \tau^{\ra}\omega + dx \wedge \tilde \delta^{\ra} \omega )\\
&=\frac{1}{\sqrt 2}\tilde \Phi^{-1}\iu (-dx \wedge \tilde \tau^{\ra}\tilde \delta^{\ra} \omega - \tilde \tau^{\ra}\tilde \delta^{\ra}\tilde \tau^{\ra}\omega  - dx \wedge \tilde \delta^{\ra}\tilde \tau^{\ra}\omega - \tilde \delta^{\ra} \omega)\\
&=\iu(-\tilde \tau^{\ra}\tilde \delta^{\ra} - \tilde \delta^{\ra}\tilde \tau^{\ra})\omega=\tilde \D^{\ra}\omega \ .
\end{align*}  

We conclude that $$(\tilde \D^Z)^+=c(dx)\left(\ra_x-\tilde \Phi \tilde \D^{\ra} \tilde \Phi^{-1}\right) \ .$$
Here, as usual, $(\tilde \D^Z)^+:\frac 12(1+\tilde \tau^Z)\Omega^*_{(2)}(Z,\F_Z) \to \frac 12(1-\tilde \tau^Z)\Omega^*_{(2)}(Z,\F_Z)$ is the restriction of $\tilde \D^Z$ to the positive eigenspace of $\tilde \tau^Z$. Its adjoint is denoted by  $(\tilde \D^Z)^-$.

Thus on the cylindrical end  $$\tilde \D^+=c(dx)\left(\ra_x-\tilde \Phi \tilde \D^{\ra} \tilde \Phi^{-1}\right) \ .$$
 
If $f_{\ra}$ is a homotopy equivalence, then the operator $\tilde \D^{\ra}_{\alpha_0,1}$ is invertible. By standard arguments this implies that $\tilde \D_{\alpha_0,1}^+$ is a Fredholm operator. This shows the first part of the following proposition:

\begin{prop}
\label{invert}
If $f_{\ra}$ is a homotopy equivalence, then $\tilde \D_{\alpha_0,1}^+$ is a Fredholm operator. 

If $f$ is a homotopy equivalence, then $\tilde \D_{\alpha_0,1}^+$ is invertible.
\end{prop}

For the proof of the second assertion we need the following technical lemma, which replaces \cite[Prop. 1.9]{hs}:

\begin{lem}
Assume that $D$ is a regular operator on a Hilbert module $H$ with $D^2=0$. If $\Ker D=\im D$, then the operator $D+D^*$ is invertible. In particular, it is selfadjoint and regular. 
\end{lem}

\begin{proof}
Since $\im D(1+D^*D)^{-1/2}=\im D$ is closed and $D(1+D^*D)^{-1/2} \in B(H)$, the module $\im D$ is complemented. Its complement is $\Ker D^*$.

With respect to the decomposition $H=\Ker D \oplus \Ker D^*$ it holds that

$$D+D^*=\left(\begin{array}{cc} 0 & D \\ D^* & 0 \end{array}\right)\ .$$
Here $D:\Ker D^* \to \Ker D$ and $D^*:\Ker D \to \Ker D^*$ are invertible. It follows that $D+D^*$ is invertible. 
\end{proof}

\begin{proof}[Proof of the Proposition.]
The argument is as in the proof of \cite[2.1 Lemme]{hs}: 

One checks that the proof of Lemma \ref{iso} still goes through for manifolds with cylindrical ends. It implies that $\Ker \delta_{\alpha_0}=\im \delta_{\alpha_0}$, and hence also $\Ker \tilde \delta_{\alpha_0,1}=\im \tilde \delta_{\alpha_0,1}$. Then the assertion follows from the previous Lemma.
\end{proof}

Let $\alpha,\beta \in \C(\bbbr)$ be functions as in the proof of Prop. \ref{product}.

Let $\ve_0 \in (0,\infty)$.
For $x \in \bbbr$ we define $\ve(x)=\psi(x)^{-1}\ve_0 \in [\ve_0,\infty]$, where $\psi \in \C(\bbbr)$ is an increasing function with $\psi(x)=0$ for $x\le 0.4$ and $\psi(x)=1$ for $x \ge 0.6$. Thus $\ve(x)=\infty$ for $x \le 0.4$.

\begin{prop}
Let $f_{\ra}$ be a homotopy equivalence. 

The index of the operator $\ra_x - D_x$ with $D_x:=\tilde \D^{\ra}_{\alpha(x),\beta(x),\ve(x)}$ vanishes.
\end{prop}

\begin{proof}
First note that the domain of $D_x$ is $\Hi^1_{\ra}$ and thus does not depend on $x$. Furthermore, $D_x:\Hi^1_{\ra}\to \Hi_{\ra}$ depends on $x$ in a continuous way in the norm topology. Here we use that $\ve$ is locally constant on $\supp \alpha \cup \supp \beta$. For $x \notin [0,1]$ the operator $D_x$ is invertible.
The index $\ind(\ra_x - D_x)$ is well-defined and equals $-\spfl((D_x)_{x \in [0,1]})$. 

In the following we homotop the path $(D_x)_{x \in [0,1]}$ to a path of invertible operators through paths with invertible endpoints. The strategy is as in the proof of Prop. \ref{product}. Let $\tilde \alpha$, $\tilde \beta$ be the functions defined there.

We set $D_{t,x}:=\tilde \D^{\ra}_{\tilde\alpha(t,x),\tilde \beta(t,x),\ve(x)}$. The operators $D_{t,0}$, $D_{t,1}$ are invertible for all $t \in [0,2]$.

Furthermore, the family $D_{t,x}$ defines a regular and selfadjoint operator on $C([0,2]\times [0,1],\Hi_{\ra})$ with domain $C([0,2]\times [0,1],\Hi^1_{\ra})$. Note that for $t \neq 0$ the operator $D_{t,x}$ does not depend continuously on $x$ in the operator norm topology at $x$ with $\psi(x)=0,~ \psi'(x)>0$. We use the spectral flow and not the index for our argument since for the index we would need to know that $\ra_x-D_{t,x}$ is Fredholm, which does not follow here from the general results in \cite{aw}.

The homotopy invariance of the spectral flow yields that
$$\spfl((D_{0,x})_{x \in [0,1]})= \spfl((D_{2,x})_{x \in [0,1]}) \ .$$
The family $(D_{2,x})_{x \in [0,1]}$ defines an invertible operator on $C([0,1],\Hi_{\ra})$. Thus the right hand side vanishes.
\end{proof}

Let $\chi: N \cup M \to [0,1]$ be a smooth positive function such  that $\chi(x,y)=1$ for $(x,y) \in Z^+$ with $x \ge 1$ and $\chi(x,y)=0$ for $(x,y) \in Z^+$ with $x \le \frac 12$. We also assume that $\chi|_{(N \cup M) \setminus Z^+}=0$.

We define a selfadjoint operator $\D(A(f_{\ra})_{\alpha_0,1,\ve_0})$ that is odd with respect to the grading operator $\tau=\tilde \tau_{0,0}$ by setting
$$\D(A(f_{\ra})_{\alpha_0,1,\ve_0})^+:=\D^+-c(dx) \chi \tilde \Phi_{0,0} A(f_{\ra})_{\alpha_0,1,\ve_0}\tilde \Phi_{0,0}^{-1} \ .$$
See \cite{lpdir,wazyl} for this type of regularization.

From this definition one gets the adjoint $\D(A(f_{\ra})_{\alpha_0,1,\ve_0})^-$ and thus also $\D(A(f_{\ra})_{\alpha_0,1,\ve_0})$.
If $f_{\ra}$ is a homotopy equivalence, the boundary operator $\D^{\ra} +A(f_{\ra})_{\alpha_0,1,\ve_0}$ is invertible, and thus the operator $\D(A(f_{\ra})_{\alpha_0,1,\ve_0})$ is Fredholm. 

\begin{theorem}
\label{indsign}
Assume that $f_{\ra}$ is a homotopy equivalence. Then $\D(A(f_{\ra})_{\alpha_0,1,\ve_0})^+$ is Fredholm.

If $f$ is a homotopy equivalence, then the index of $\D(A(f_{\ra})_{\alpha_0,1,\ve_0})^+$ vanishes.
\end{theorem}

Since in the proof and further on we will repeatedly use variants of the relative index theorem in \cite{bu}, we quickly recall its idea:

Assume that a closed oriented Riemannian $R$ decomposes as $R=R_1 \cup_X R_2$ for two manifolds $R_1, R_2$ with common boundary $X$ and assume that $D:\C(R,E) \to \C(R,F)$ is a first order elliptic differential operator which is of the form $u(\ra_x - D_X)$ in a tubular neighbourhood of $X$, with $D_X$ a differential operator on $\C(X,E|_X)$ and $u:E|_X \to F|_X$ a unitary isomorphism. If $D_X$ is invertible, then the induced operators $D_1,~D_2$ on $R_1,~R_2$ are Fredholm (after attaching cylindric ends) and $\ind(D)=\ind(D_1)+\ind(D_2)$. 

The theorem was proven by Bunke for Dirac operators associated to $C^*$-vector bundles over certain complete manifolds cut by a compact hypersurface. The proof extends in a straightforward way to more general Fredholm operators which have the structure $u(\ra_x-D_X)$ near the compact hypersurface $X$. Note that if $R_1$ is a cylinder and $D_1$ translation invariant, then $\ind(D_1)=0$.

\begin{proof}
It remains to prove the second claim. Let $D^{ip}$ be the selfadjoint regular Fredholm operator on $\Hi$ which is odd with respect to the grading operator $\tilde \tau_{\alpha_0,1}$ and which equals $\tilde \D_{\alpha_0,1}$ on $(N \cup M) \setminus Z^+$
and on the cylindrical end $Z^+$ fulfills
$$(D^{ip})^+=c(dx)\left(\ra_x-\tilde \Phi_{\alpha_0,1}\tilde \D^{\ra}_{\alpha(x),\beta(x),\ve(x)}\tilde \Phi_{\alpha_0,1}^{-1}\right) \ .$$

By a variant of the relative index theorem in \cite{bu} 
$$\ind((D^{ip})^+)=\ind((\tilde \D_{\alpha_0,1})^+)+ \ind(\ra_x- \tilde \D^{\ra}_{\alpha(x),\beta(x),\ve(x)}) \ .$$

Then the previous two propositions imply that $\ind((D^{ip})^+)=0$.

Now we define a path $(D_t^{ip})_{t\in [0,2]}$ of selfadjoint regular Fredholm operators on $\Hi$ with $D_0^{ip}=D^{ip}$ and $D_2^{ip}=\D(A(f_{\ra})_{\alpha_0,1,\ve_0})$. 

Let $\ov \alpha,\ov \beta: [0,2]\times [0,\infty) \to  \bbbr$ be given by
\begin{align*}
\ov\alpha(t,x)&=\left\{\begin{array}{ll} \alpha(x) & t=0, x \in [0,\infty);~ t\in (0,2], x >0.5\\
(1-t)\alpha(x) & t\in (0,1], x \le 0.5\\
0 &  t=(1,2], x \le 0.5 \ , 
\end{array} \right.\\
\ov \beta(t,x) &=\left\{\begin{array}{ll} \beta(x) & t \in [0,1], x \in [0,\infty);~ t \in (1,2], x >0.5\\
(1-t)\beta(x) & t\in (1,2], x \le 0.5 \ . 
\end{array}\right.
\end{align*}

The functions are continuous since $\alpha(0.5)=\beta(0.5)=0$. Note that $\ov{\beta}|_{\supp\ov \alpha}=1$.

The path of operators $(D_t^{ip})_{t\in [0,2]}$ is defined by requiring that $D_t^{ip}$ is odd with respect to $\tilde\tau_{\ov\alpha(t,0),\ov\beta(t,0)}$ and that on $(N \cup M) \setminus Z^+$
$$D^{ip}_t=\tilde \D_{\ov\alpha(t,0),\ov\beta(t,0)}$$ 
and on $Z^+$
$$(D^{ip}_t)^+=c(dx)\left(\ra_x-\tilde\Phi_{\ov\alpha(t,0),\ov\beta(t,0)}\tilde \D^{\ra}_{\ov\alpha(t,x),\ov\beta(t,x),\ve(x)}\tilde\Phi_{\ov\alpha(t,0),\ov\beta(t,0)}^{-1}\right) \ .$$
The operator $D^{ip}_t$ has domain $\Hi^1$ and depends in a continuous way on $t$ as a bounded operator from $\Hi^1$ to $\Hi$. Note that on the cylindrical end with $x>2$ it does not depend on $t$. Thus we can construct a local parametrix there which is independent of $t$. A local parametric on the cylindrical piece with $x \in [0,3]$ may be constructed as in \cite{aw}. One may use a doubling construction to get a local parametrix on the remaining compact part: The doubled operator has compact resolvents and thus is Fredholm.

Hence $(D^{ip}_t)_{t \in [0,2]}$ is Fredholm as an operator from $C([0,2],\Hi^1)$ to $C([0,2],\Hi)$. It follows that the odd operator $(D^{ip}_t)_{t \in [0,2]}$ on the $\bbbz/2$-graded space $C([0,2],\Hi)$ with grading operator $(\tilde\tau_{\ov\alpha(t,0),\ov\beta(t,0)})_{t\in [0,2]}$ defines an element in $KK_0(\bbbc,C([0,2],\A))$, see \cite[\S 2.1]{wancsf}. The homotopy invariance of $KK$-theory implies that
$$0=\ind((D^{ip}_0)^+)=\ind((D^{ip}_2)^+) =\ind(\D(A(f_{\ra})_{\alpha_0,1,\ve_0}^+) \ .$$ 
\end{proof}

All results also hold with $-\alpha_0$ instead of $\alpha_0$.

Before returning to $\rho$-invariants we apply the result to higher signatures and  $L^2$-signatures for manifolds with boundary. (We tacitly use that Atiyah--Patodi--Singer index problems for manifolds with boundary can be translated into index problems for manifolds with cylindrical ends.) 

Consider the situation in \S \ref{high}. The definition of higher signature classes for manifolds with cylindrical ends is due to Leichtnam and Piazza \cite{lpsign,lpdir}; see \cite{waprod} for the variant used here. The higher signature class $\sigma(M,\F_M)$ of $M$ is defined if the range of the closure of $d:\Omega^{m-1}(\ra M,\F_M)\to \Omega_{(2)}^m(\ra M, \F_M),~m=\dim M/2$ is closed. It equals $\ind(\D_M(A))$, where $A$ is a bounded selfadjoint operator vanishing on $V_{\ra M}\subset \Omega_{(2)}^*(\ra M,\F_M)$ and anticommuting with the operator $\inv_{\ra M}$. Here the operator $\D_M(A)$ is constructed from $\D_M$ and $A$ in analogy to the construction of $\D(A(f_{\ra})_{\alpha_0,1,\ve_0})$ from $\D$ and $A(f_{\ra})_{\alpha_0,1,\ve_0}$ given above. 

The homotopy invariance of a conical signature class defined in the same situation but using a different regularization was proven in \cite{llp}. Both signature classes agree, see \cite[Prop. 9.1]{waprod}. The following corollary gives a direct proof of the homotopy invariance of the cylindrical signature class.

\begin{cor}
Let $f:M \to N$, $\F_M$, $\F_N$ be as in the beginning of this section. 

Assume that the range of the closure of $d:\Omega^{m-1}(\ra N \cup \ra M,\F)\to \Omega_{(2)}^m(\ra N\cup \ra M, \F)$ with $m=\dim M/2$ is closed. 

If $f$ is a homotopy equivalence, then $\sigma(M,\F_M)=\sigma(N,\F_N)$.
\end{cor}

\begin{proof}
It holds that $\sigma(N,\F_N)-\sigma(M,\F_M)=\sigma(N \cup M^{op},\F)$.

By the proof of Prop. \ref{prophigh} for $\ve_0>0$ small enough 
$$\sigma(N \cup M^{op},\F)=\ind(\D(A(f_{\ra})_{\alpha_0,1,\ve_0})^+) \ .$$
Now the assertion follows from the theorem.
\end{proof}

In a similar way one obtains an analytic proof of the homotopy invariance of the signature of a manifold with boundary associated to a trace $\nu$ as in \S \ref{L2}. It can be defined as
$$\sign_{\nu}(M):=\nu(1)\int_M L(M)- \eta_{\nu}(\D_{\ra M})\ .$$
A topological proof of its homotopy invariance for $\nu=\nu_{e}$ follows from the results in \cite{ls}. 

\begin{cor}
With the definitions from the beginning of this section it holds that $\sign_{\nu}(M)=\sign_{\nu}(N)$ if $f:M\to N$ is a homotopy equivalence.
\end{cor}

\begin{proof}
By the previous theorem and the Atiyah--Patodi--Singer index theorem over $C^*$-algebras 
\begin{align*}
0&=\frac 12\Bigl(\nu\ind(\D(A(f_{\ra})_{\alpha_0,1,\ve_0})^+)+\nu\ind(\D(A(f_{\ra})_{-\alpha_0,1,\ve_0})^+)\Bigr)\\
&=\nu(1)\int_N L(N)-\nu(1)\int_M L(M)-\nu (\eta(f_{\ra}))\ .
\end{align*}
By Prop. \ref{L2prop}
$$\nu(\eta(f_{\ra}))=\eta_{\nu}(\D_{\ra N})-\eta_{\nu}(\D_{\ra M}) \ .$$
Thus $$\nu(1)\int_N L(N)-\nu(1)\int_M L(M)-\nu (\eta(f_{\ra}))=\sign_{\nu}(N)-\sign_{\nu}(M) \ .$$
\end{proof}

Now we turn to the structure set:
Let $N$ be an odd-dimensional closed oriented connected manifold. The surgery structure set ${\mathcal S}(N)$ consists of pairs $(M,f)$, where $M$ is a closed oriented manifold and $f:M \to N$ is an orientation preserving homotopy equivalence. Two pairs $(M_0,f_0)$, $(M_1,f_1)$ are equivalent if there is an oriented manifold $W$ with boundary $\ra W=M_1 \cup M_0^{op}$ and an orientation preserving homotopy equivalence $F:W \to [0,1] \times N$ of manifolds with boundaries such that $F|_{M_i}=f_i,~i=0,1$. 

Let $\tilde N \to N$ be a Galois covering with deck transformation group $\Gamma$.

We let $\A$, $\A_{\infty}$ be as in \S \ref{rhodef} and set $\F_N:=\tilde N \times_{\Gamma} \A$. With these data we get for each pair $(M,f)$ a noncommutative $\rho$-form
$\rho(f)$.

\begin{cor}
\label{corrhoS}
The $\rho$-form yields a well-defined map $$\rho^{{\mathcal S}}:{\mathcal S}(N) \to \Oi\Ai/\ov{[\Oi\Ai,\Oi\Ai]_s + \di\Oi\Ai+\Oi^{<e>}\Ai}$$
$$[(M,f)]\mapsto \rho(f) \ .$$
\end{cor}

\begin{proof}
Let $(M_0,f_0)$, $(M_1,f_1)$ be two equivalent pairs.
We apply the theorem to the homotopy equivalence $F:W \to [0,1]\times N$. One calculates the Chern character of the index in $\Oi\Ai/\ov{[\Oi\Ai,\Oi\Ai]_s + \di\Oi\Ai+\Oi^{<e>}\Ai}$ via the Atiyah--Patodi--Singer index theorem over $C^*$-algebras. The local term vanishes since we divide out the forms $\Oi^{<e>}\Ai$ localized at the identity. The contribution from the boundary components is $\rho(f_1)-\rho(f_0)$. Since the index vanishes, it also vanishes.
\end{proof}

\section{Compatibility with the $L$-group action}
\label{compat}

Assume that $n=\dim N \ge 4$ and $\Gamma=\pi_1(N)$.
Now we study how the action of $L_{n+1}(\bbbz\Gamma)$ on $\surg(N)$ combines with $\rho^{{\mathcal S}}$. 
We refer to the literature (for example \cite[\S 5.3]{lu}) for its definition and only recall the relevant features here. 

Let $[(M,f)]\in \surg(N)$.
 
An element $a \in L_{n+1}(\bbbz\Gamma)$ can be represented by a normal map of manifolds with boundary $F:W \to [0,1] \times M$ of degree one between a manifold $W$ with two boundary components $\ra_0 W, ~\ra_1 W$ and $[0,1]\times M$ such that $\ra_0 F:=F|_{\ra_0 W}:\ra_0 W \to \{0\} \times M$ is a diffeomorphism and $\ra_1 F:=F|_{\ra_1 W}:\ra_1 W \to \{1\}\times M$ is a homotopy equivalence. (We do not need the normality condition in the following and therefore omit the bundle data.) Then $$a [(M,f)]=[(\ra_1 W,f \circ \ra_1 F)] \ .$$   
By Prop. \ref{propadd} 
\begin{align*}
\rho^{{\mathcal S}}(a[(M,f)])-\rho^{{\mathcal S}}([(M,f)])&=\rho^{{\mathcal S}}([(\ra_1 W,f \circ \ra_1 F)])-\rho^{{\mathcal S}}([(M,f)])\\
&=\rho(f \circ \ra_1 F)-\rho(f)\\
&=\rho(\ra_1 F)-\rho(\id_M) \\
&=\rho(\ra_1 F)-\rho(\ra_0 F) \ .
\end{align*} 

By Theorem \ref{indsign} the operator $\D(A(F|_{\ra W})_{\pm \alpha_0,1,\ve_0}$ is Fredholm since $F|_{\ra W}$ is a homotopy equivalence. Thus, the definition of the map $\sign^L$ in the following theorem makes sense.

\begin{theorem}
\label{commute}
Let $n \ge 4$.
The map $$\sign^L\colon L_{n+1}(\bbbz\Gamma) \to K_*(\A)\ten \bbbz[\frac 12]\ ,$$ 
$$[F:W \to [0,1] \times M] \mapsto \frac 12\Bigl(\ind (\D(A(F|_{\ra W})_{\alpha_0,1,\ve_0})^+)+\ind (\D(A(F|_{\ra W})_{-\alpha_0,1,\ve_0})^+ \Bigr)$$

is well-defined. In $\Oi\Ai/\ov{[\Oi\Ai,\Oi\Ai]_s + \di\Oi\Ai+\Oi^{<e>}\Ai}$ it holds that

$$\rho^{{\mathcal S}}(a[(M,f)]) - \rho^{{\mathcal S}}([(M,f)]) =-\ch \sign^L(a) \ .$$  

Thus, the diagram 
$$\xymatrix{
L_{n+1}(\bbbz \Gamma) \ar[d]^{-\sign^L} \ar[r] & {\mathcal S}(N)\ar[d]^{\rho^{\mathcal S}}\\
K_*(\A)\ten \bbbz[\frac 12] \ar[r]^-{\ch} &\Oi\Ai/\ov{[\Oi\Ai,\Oi\Ai]_s + \di\Oi\Ai+\Oi^{<e>}\Ai}}$$
commutes.
\end{theorem}

Note that the diagram is slightly sloppy: If we interpret the first line as the map $L_{n+1}(\bbbz \Gamma) \to {\mathcal S}(N), ~a \mapsto a[(M,f)]$ for some $(M,f) \in {\mathcal S}(N)$, then the vertical map on the right hand side is $\rho^{{\mathcal S}}-\rho^{{\mathcal S}}([(M,f)])$ in order to make the diagram commute. (Same for $(M,f)=(N,\id_N)$ since we do not know whether $\rho^{{\mathcal S}}([(N,\id_N)])$ vanishes.)

\begin{proof}
We only have to show that $\sign^L$ is well-defined.
The compatibility formula follows from the previous calculations and the Atiyah--Patodi--Singer index theorem over $C^*$-algebras.

We consider a more general situation, which will be relevant also for our applications: Let $V$, $W$ be $n+1$-dimensional oriented connected manifolds with boundary and assume given a reference map $V \to B\Gamma$. Furthermore, let $F\colon W \to V$ be a normal map of degree one between oriented manifolds with boundary such that $F|_{\ra W}:\ra W \to \ra V$ is a homotopy equivalence. By \cite[Ch. 9]{wall} we get an element $[F\colon W \to V]$ in a group $L^1_{n+1}(B\Gamma)$ isomorphic to $L_{n+1}(\bbbz\Gamma)$. (To be precise, Wall considered the simple version of $L$-theory, however the result translates easily to the $h$-decorated version used here.) 

The reference map $V\to B\Gamma$ induces $\Gamma$-principal bundles on $V$, $W$, and thus we can define 
$$\sign_+^L(F\colon W \to V):=\ind (\D(A(F|_{\ra W})_{\alpha_0,1,\ve_0})^+) \in K_*(\A)\ .$$  (Tacitly, we choose metrics and glue cylindrical ends to the boundaries whenever necessary.)

In the following we consider representatives $(F_i\colon W_i \to V_i),~i=0,1$ such that $[F_0\colon W_0 \to V_0]=[F_1\colon W_1 \to V_1]$ in $L^1_{n+1}(B\Gamma)$. We show that 
$$\sign_+^L(F_0\colon W_0 \to V_0)=\sign_+^L(F_1\colon W_1 \to V_1) \ .$$ 
Since the proof carries over to $\sign_-^L(F\colon W \to V):=\ind (\D(A(F|_{\ra W})_{-\alpha_0,1,\ve_0})^+)$,
this will establish the assertion.

As before Prop. \ref{invert} we define an operator $\tilde \D_{\alpha_0,1}^{F_i}$ associated to $F_i$. By the proposition it is Fredholm since $F_i|_{\ra W_i}$ is a homotopy equivalence. The index of $\tilde \D_{\alpha_0,1}^{F_i}$ may be nonzero since in general $F_i$ is not a homotopy equivalence. The calculations in the previous section imply that $\ind (\D(A(F_i|_{\ra W_i})_{\alpha_0,1,\ve_0})^+)=\ind ((\tilde \D_{\alpha_0,1}^{F_i})^+)$. In the following we will use the latter expression as a definition of $\sign_+^L(F_i\colon W_i \to V_i)$.

By the arguments in the proof of Theorem 9.4 in \cite{wall} one may attach $1$- and $2$-handles in the interior of $W:=W_0 \cup W_1^{op}$ and $V:=V_0 \cup V_1^{op}$ and extend the map $F:=F_0\cup F_1$ and the reference maps such that for the resulting normal map $(F'\colon  W' \to V')$ the spaces $V', W'$ are connected and the reference map $V' \to B\Gamma$ induces an isomorphism on fundamental groups. Since $[F'\colon  W' \to V']=0$  in $L^1_{n+1}(B\Gamma)$,  one may obtain a homotopy equivalence $(F''\colon W'' \to V')$ by further surgery steps in the interior of $W'$, see for example \cite[Theorem 4.47(3)]{lu}. Here one needs that $n \ge 4$.

Recall that one defines an oriented manifold tried as a triple $(Y,\ra_+ Y, \ra_- Y)$, where $Y$ is an oriented manifold with boundary and $\ra Y$ decomposes as the union of oriented manifolds with boundary $\ra_+Y$ and $\ra_-Y$ with $\ra \ra_+Y = \ra \ra_-Y$.

Since $(F''\colon W'' \to V')$ is constructed from $(F\colon W \to V)$ by surgery leaving the boundary fixed and thus is in the same normal bordism class (see \cite[Theorem 3.59]{lu} and the discussion in \cite[\S 4.7.1]{lu}), there is a normal map of degree one between oriented manifold triads $G\colon (Y,\ra_+ Y, \ra_- Y) \to (X,\ra_+X,\ra_- X)$ and a reference map $X \to B\Gamma$ such that $(G|_{\ra_+ Y}\colon \ra_+ Y \to  \ra_+ X)=(F\cup F''\colon  W \cup (W'')^{op} \to V\cup (V')^{op})$ and the reference maps for $\ra X_+$ and $V \cup (V')^{op}$ agree, too. Namely, $X, ~Y$ are the traces of the surgeries. Furthermore, $\ra_- X=[0,1] \times \ra V,~\ra_- Y=[0,1] \times \ra W$ and $G|_{\ra_- Y}\colon \ra_- Y \to \ra_- X$ equals $\id \times F|_{\ra W}$ and thus is a homotopy equivalence. 

Now the general strategy of the proof is analogous to the one in the positive scalar curvature case, see \cite[\S 1.4]{bu}. It uses the cobordism invariance of the index and the relative index theorem to show that $\sign_+^L(G|_{\ra_+ Y}\colon \ra_+ Y\to \ra_+ X)=0$ and thus $$\sign_+^L(F\colon W \to V)=\sign_+^L(F''\colon W'' \to V') \ .$$ 
Prop. \ref{invert} implies that $\sign_+^L(F''\colon W'' \to V')=0$ since $F''$ is a homotopy equivalence.
By
$$\sign_+^L(F\colon W\to V)=\sign_+^L(F_0\colon W_0\to V_0)-\sign_+^L(F_1\colon W_1\to V_1)$$ this implies that $\sign_+^L$ is well-defined.

We check the details.
Since $(G\colon Y \to X)$ is a degree one map between manifolds with boundary, we can define the operator 
$$\tilde \D^G_{\alpha_0,1}:=(-\iu)(\tilde \delta_{\alpha_0,1}^G + \tilde \tau_{\alpha_0,1}^G\tilde \delta_{\alpha_0,1}^G \tilde \tau_{\alpha_0,1}^G) \ .$$
This is the analogue of the operator in eq. (\ref{hsop}) for odd dimensions.  The following discussion is closely related to the discussion in the classical case in \cite[\S 5.2]{waprod}.

Note that in general the operator $\tilde \D^G_{\alpha_0,1}$ is not Fredholm since $G|_{\ra Y}$ need not be a homotopy equivalence. 

The operator commutes with the involution $\tilde \tau_{\alpha_0,1}^G$. In the following we restrict to the positive eigenspace of $\tilde \tau_{\alpha_0,1}^G$ and determine the boundary operator by adapting the calculations from the previous section. (We omit the indices.) 

On the cylindrical end $\tau=\tau_1\tau_{\ra}$ with $\tau_1=\iu c(dx)$.

We decompose $\tilde \tau^{\ra}=t^{o}+ \gamma_{\ra}t^{e}$. By an analogue of Lemma \ref{changeinvol} on the cylindrical end $\tilde \tau=\tau_1(t^{o}+\gamma_{\ra} t^e)=\tau_1 \tilde \tau^{\ra}$. 

As before, we also decompose $\tilde \delta^{\ra}=d^{o}+ \gamma_{\ra} d^{e}$ and define $\tilde \delta^{cs}:=\tilde \delta-\iu dx\, \ra_x=\iu(\gamma_1 d^o + \gamma d^e)=\iu \gamma_1 \tilde \delta^{\ra}$.

On the cylindrical end
\begin{align*}
\tilde \D^G &=c(dx)\left(\ra_x -  (-\iu) c(dx)(\tilde\delta^{cs}+ \tilde \tau\tilde \delta^{cs}\tilde \tau)\right)\\
&=c(dx)\left(\ra_x -   c(dx)\gamma_1(\tilde \delta^{\ra} -  \tilde \tau^{\ra} \tilde \delta^{\ra} \tilde \tau^{\ra})\right) \ .
\end{align*}

Recall from the previous section that for a form $\omega$ on the boundary
$$\tilde \Phi(\omega)=\frac{1}{\sqrt 2}(dx \wedge \omega + \tilde \tau (dx\wedge \omega))=\frac{1}{\sqrt 2}(dx \wedge \omega - \iu\tilde \tau_{\ra} \omega) \ .$$ 

Thus 
\begin{align*}
\lefteqn{\tilde \Phi^{-1} c(dx)\gamma_1(\tilde \delta^{\ra} -  \tilde \tau^{\ra} \tilde \delta^{\ra} \tilde \tau^{\ra})\tilde \Phi(\omega)}\\
&=\frac{1}{\sqrt 2}\tilde \Phi^{-1} c(dx)\gamma_1(\tilde \delta^{\ra} -  \tilde \tau^{\ra} \tilde \delta^{\ra} \tilde \tau^{\ra})(dx \wedge \omega -\iu\tilde \tau^{\ra}\omega) \\
&=\frac{1}{\sqrt 2}\tilde \Phi^{-1}\bigl((\tilde \delta^{\ra} -  \tilde \tau^{\ra} \tilde \delta^{\ra} \tilde \tau^{\ra})\omega-\iu dx \wedge (\tilde \delta^{\ra}\tilde \tau^{\ra} -  \tilde \tau^{\ra} \tilde \delta^{\ra})\omega\bigr)\\
&=-\iu(\tilde \delta^{\ra}\tilde \tau^{\ra} -  \tilde \tau^{\ra} \tilde \delta^{\ra})\omega =\iu\tilde \tau^{\ra}\tilde \D^{\ra} \ .
\end{align*}

Here $\tilde \D^{\ra}:=\tilde \delta^{\ra} -  \tilde \tau^{\ra} \tilde \delta^{\ra} \tilde \tau^{\ra}=\tilde \D_{\alpha_0,1}^{G|_{\ra Y}}$.

Thus on the cylindrical end and restricted to the positive eigenspace of $\tilde \tau$
$$\tilde \D^G =c(dx)\left(\ra_x -   \tilde \Phi \iu\tilde \tau^{\ra}\tilde \D^{\ra} \tilde \Phi^{-1}\right) \ .$$

It follows that $\iu\tilde \tau^{\ra}\tilde \D^{\ra}$ is the boundary operator of $\tilde \D^G$ (up to the isomorphism $\tilde \Phi$). Its grading operator is given by $\iu \tilde \Phi^{-1}c(dx)\tilde \Phi=\tilde\tau^{\ra}$.

By cobordism invariance the index of $(\iu\tilde \tau^{\ra}\tilde \D^{\ra})^+$ vanishes. (See \cite{hils} for a very general proof of the cobordism invariance.)  Thus also the index of $(\tilde \D^{\ra})^+=(\tilde \D_{\alpha_0,1}^{G|_{\ra Y}})^+$ vanishes.

By (an analogue of) the relative index theorem from \cite{bu}
$$\ind ((\tilde \D_{\alpha_0,1}^{G|_{\ra Y}})^+)=\ind ((\tilde \D_{\alpha_0,1}^{G|_{\ra_+Y}})^+) + \ind ((\tilde \D_{\alpha_0,1}^{G|_{\ra_-Y}})^+) \ .$$
Since $G|_{\ra_-Y}$ is a homotopy equivalence, it holds that $\ind ((\tilde \D_{\alpha_0,1}^{G|_{\ra_- Y}})^+)=0$.

Thus 
$$\sign_+^L(G|_{\ra_+ Y}\colon \ra_+ Y\to \ra_+ X)=\ind ((\tilde \D_{\alpha_0,1}^{G|_{\ra_+Y}})^+)=0 \ .$$

This concludes the proof of the well-definedness.
\end{proof}

\section{Applications}
\label{appl}

In the following we give some typical applications which serve to illustrate how the product formula (Theorem \ref{prodtheo}) can be used to get information on the structure set of products. These are mainly interesting for groups with torsion; 
the results in \cite{ps} imply that the map $\rho^{\mathcal S}$ is trivial for torsionfree groups for which the maximal Baum--Connes conjecture holds.

Let $N_1$ and $N_2$ be oriented closed connected manifolds of odd and even dimension, respectively. We set $\Gamma_i= \pi_1(N_i)$ and use the universal coverings. In the following the notation is as in \S \ref{prodform} (with $N_2=M_2$).

The main difficulty is to find an appropriate projective system $(\A_i)_{i \in \bbbn_0}$. The following two situations are particularly simple:
\begin{enumerate} 
\item $\B$ is finite-dimensional.
\item $\B=C(T^k)$ and the image of the map $\bbbc\Gamma_1 \to C(T^k)$ is in $\C(T^k)$.
\end{enumerate}
Here one may replace $\B$ by $\Ca$ and $\Gamma_1$ by $\Gamma_2$.

We consider the case where (1) or (2) holds for $\B$. We choose a projective system $(\Ca_i)_{i \in \bbbn_0}$ as in \S \ref{rhodef}. 
If (1) holds we set $\A_i:=\B \ten \Ca_i$. Since $\B$ is finite-dimensional, the tensor product here is algebraic, and $\A_i$ is closed under holomorphic functional calculus in $\A$.

In the case of (2) we set $\A_i=\cap_{j+l=i} C^j(T^k,\Ca_l)$. 

The following proposition is a straightforward application of the product formula and of Cor. \ref{corrhoS}. Here $\B=M_m(\bbbc)$. Thus we are in situation (1).

\begin{prop}
Let $h_1,h_2: \Gamma_1 \to U(m)$ be homomorphisms and $\rho_{h_1,h_2}$ the induced Atiyah--Patodi--Singer $\rho$-invariant.
Let $[(M_a,f_a)], [(M_b,f_b)] \in \surg(N_1)$ with $\rho_{h_1,h_2}(M_a) \neq \rho_{h_1,h_2}(M_b)$. Furthermore, assume that $\ch(\ind(\D_{N_2})) \neq 0$. Then  
the elements  $[(M_a \times N_2,f_a \times \id)], [(M_b \times N_2,f_b \times \id)] \in \surg(N_1 \times N_2)$ are not equal.  
\end{prop}

A criterion for the nonvanishing of $\ch(\ind(\D_{N_2}))$ for an appropriate system $(\Ca_i)_{i \in \bbbn_0}$ can be deduced from results by Connes and Moscovici:
Recall that for any (alternating $\Gamma_2$-invariant) group cocycle $\tau$ on $\Gamma_2$ there is an associated reduced cyclic cocycle $c_{\tau}$ on $\bbbc\Gamma_2$ . Let $\Ca=C^*_r\Gamma_2$ and let $(\Ca_i)_{i \in \bbbn_0}$ be the projective system associated to the Connes--Moscovici algebra (\cite{cm}, see for example \cite[pp. 337f.]{waindfor} for details as needed here). If $c_{\tau}$ extends to a continuous cyclic cocycle on $\Ca_i$, then the higher signature of $N_2$ associated to $\tau$ equals $c_{\tau} \ch(\ind(\D_{N_2}))$ by the higher index theorem of Connes and Moscovici. Thus in this case the nonvanishing of the higher signature of $N_2$ associated to $\tau$ implies the nonvanishing of $\ch(\ind(\D_{N_2}))$. By the results of \cite{cm}, for $\Gamma_2$ Gromov hyperbolic the existence of a nonvanishing higher signature implies the nonvanishing of $\ch(\ind(\D_{N_2}))$.

We get a similar proposition for the $L^2$-$\rho$-invariant, however we have to impose additional technical conditions on the groups. Recall from \cite[pp. 384f.]{cm}: 
A group $\Gamma$ endowed with a word length $| \cdot |$ 
\begin{itemize}
\item
has property Rapid Decay (RD) if there are $k \in \bbbn, C>0$ such that for $a =\sum_g a_g g\in \bbbc\Gamma$
$$\|a\|^2_{C_r^*\Gamma} \le C\sum_{g\in \Gamma} (1+|g|)^{2k}|a_g|^2\ .$$
\item
has property Polynomial Cohomology (PC) if for any complex group cohomology class there is a representative of polynomial growth.
\end{itemize}

Gromov hyperbolic groups have property (RD) and (PC). Groups with property (RD) have been extensively studied recently, see for example \cite{cr} and references therein. Less is known about property (PC). See for example \cite{m} for recent results.

\begin{prop}
Assume that $[(M_a,f_a)], [(M_b,f_b)] \in \surg(N_1)$ with $\rho_{(2)}(M_a) \neq \rho_{(2)}(M_b)$ and that one of the following conditions holds:

\begin{enumerate}
\item  $N_2=T^k$.

\item $\Gamma_1$, $\Gamma_2$ have property (RD) and $\Gamma_2$ has in addition property (PC). Furthermore, $N_2$ has a non-vanishing higher signature. 
\end{enumerate}

Then the elements  $[(M_a \times N_2,f_a \times \id)], [(M_b \times N_2,f_b \times \id)] \in \surg(N_1 \times N_2)$ are not equal. 
\end{prop}

\begin{proof}
1) Note that $C^*\bbbz^k=C(T^k)$. We take $\A_i=C^i(T^k,\B)$ with $\B=C^*\Gamma_1$. Then $\Ca_i=C^i(T^k)$. The $k$-torus has a nonvanishing higher signature $c(\ch(\ind(\D_{T^k}))$ corresponding to the reduced cyclic cocycle $$c:(\C(T^k))^{\ten^{k+1}} \to \bbbc, ~(h_0, \dots ,h_k)\mapsto \int_{T^k} h_0dh_1 \dots dh_k \ .$$
The cocycle extends to a continuous map $$\tilde c:(\C(T^k,\B))^{\ten^{k+1}} \to \B/\ov{[\B,\B]}, ~(h_0, \dots ,h_k)\mapsto \int_{T^k} h_0dh_1 \dots dh_k \ .$$
It holds that $(\nu_{e}-\nu_1) \circ \tilde c=(\nu_{e}-\nu_1)\# c$. Thus $(\nu_{e}-\nu_1)\# c$ is a continuous cyclic cocycle on $\Ai$. By the product formula we have that
$$((\nu_{e}-\nu_1)\# c)(\rho^{{\mathcal S}}([(M_i\times T^k,f_i\times \id)]))=\rho_{(2)}(M_i)c(\ch(\ind(\D_{T^k})),~i=a,b \ .$$
Thus $$\rho^{{\mathcal S}}([(M_a\times T^k,f_a\times \id)]) \neq \rho^{{\mathcal S}}([(M_b\times T^k,f_b\times \id)]) \ .$$

2) First note that $\Gamma=\Gamma_1\times \Gamma_2$ also has property (RD) with respect to the length induced by the lengths of $\Gamma_1$, $\Gamma_2$ (whose choice is implicit in our assumption). 

We take $\B=C^*\Gamma_1$, $\Ca=C^*\Gamma_2$. We have canonical maps $\pi:\A \to C^*_r \Gamma = C_r^*\Gamma_1 \ten C^*_r\Gamma_2$ and $\pi_2:\A \to C_r^*\Gamma_2$. We take $\Ai$ as the intersection of the pullback of the Connes--Moscovici algebra of $\Gamma$ with the pullback of the Connes--Moscovici algebra of $\Gamma_2$. (In addition to the seminorms induced by the pullbacks we have to impose on $\Ai$ the norm of $\A$ to make the embedding $\Ai \to \A$ continuous.)

Using ideas of Jolissaint, Connes and Moscovici proved \cite[p. 385]{cm}, that for any group cocycle $\tau$ on $\Gamma_2$ of polynomial growth there is an associated reduced cyclic cocycle $c_{\tau}$ on $\bbbc\Gamma_2$ which extends to a continuous cyclic cocycle on $\Ca_{\infty}$. Here one needs property (RD). Since $\Gamma_2$ has property (PC) and a nonvanishing higher signature, there is such a $\tau$ with $c_{\tau} \ch(\ind(\D_{N_2})) \neq 0$. 
It remains to show that $(\nu_{e}-\nu_1)\# c_{\tau}$ extends to $\Ai$. 

First consider $\nu_1 \# c_{\tau}$: This is the pullback of the cyclic cocycle $c_{\tau}$ with respect to the projection $\pi_2:\bbbc\Gamma \to \bbbc \Gamma_2$. Since $c_{\tau}$ extends to the Connes--Moscovici algebra of $\Gamma_2$, the cocycle $\nu_1 \# c_{\tau}=\pi_2^*c_{\tau}$ extends to $\Ai$. 

Now consider $\nu_{e}\# c_{\tau}$: The group cocycle $\pi_2^*\tau$ on $\Gamma$ also has polynomial growth. Here $\pi_2:\Gamma \to \Gamma_2$. Hence $c_{\pi_2^*\tau}=\nu_{e} \# c_{\tau}$ extends to $\Ai$.
\end{proof}

By generalizing the methods of \cite{cw} we get the following result, which is in the spirit of results of Leichtnam and Piazza for manifolds with positive scalar curvature \cite{lpposscal}:

\begin{prop}
Assume that $N$ is a closed oriented connected odd-dimensional manifold whose fundamental group is a product $\Gamma=\Gamma_1 \times \Gamma_2$. We assume that the following conditions hold:

\begin{enumerate} 
\item $\Gamma_1$ contains a nontrivial torsion element,
\item there are $k, m \in \bbbn$ with $k \ge 2$ and $\dim N+1-m=4k$ such that $H^m(B\Gamma_2,\bbbq) \neq 0$,
\item  $\Gamma_2=\bbbz^m$ or $\Gamma_1$, $\Gamma_2$ have property (RD) and $\Gamma_2$ has in addition property (PC). 
\end{enumerate}

Then $\surg(N)$ is infinite.
\end{prop}

\begin{proof}
The definition of $\A,~\Ai$ is as in the proof of the previous proposition.

By the compatibility formula in Theorem \ref{commute} it is enough to find an element in $L_{\dim N+1}(\bbbz\Gamma)$ which is not in the kernel of the map 
$$\ch\circ \sign^L:L_{\dim N+1}(\bbbz\Gamma) \to \Oi\Ai/\ov{[\Oi\Ai,\Oi\Ai]_s + \di\Oi\Ai+\Oi^{<e>}\Ai} \ .$$

In the following we use the description of the $L$-theory groups from the proof of Theorem \ref{commute}.  
The map $\sign_{(2)}-\sign:L_{4k}(\bbbz\Gamma_1) \to \bbbc$ is nonzero \cite{cw}. Let $[F: W \to V]\in L_{4k}(\bbbz\Gamma_1)$ be an element not in the kernel of this map. Thus $\rho_{(2)}(\ra W)-\rho_{(2)}(\ra V) \neq 0$. 

Let $c \in H^m(B\Gamma_2,\bbbq)$ be nontrivial. In the following, $\Omega_*$ denotes oriented bordism. Since the map
$$\Omega_m(B\Gamma_2)\ten \bbbq \to H_m(B\Gamma_2,\bbbq) \ , ~ [f\colon M \to B\Gamma_2] \mapsto f_*[M]$$ 
is surjective (see for example \cite[p. 205]{da}), there is $[f:N_2 \to B\Gamma_2] \in \Omega_m(B\Gamma_2)$ such that $c \cap f_*[N_2] \neq 0$. Note that $c \cap f_*[N_2]$ is a higher signature of $N_2$. 
  
Then $[(F\times \id_{N_2}):W \times N_2 \to V \times N_2]$ defines an element in $L_{\dim N+1}(\bbbz\pi_1(N))$. Its image under the above map equals
\begin{align*}
\ch\circ\sign^L [(F\times \id_{N_2}):W \times N_2 \to V \times N_2] &=-\rho(F|_{\ra W}\times \id_{N_2}) \\
&=-\rho(F|_{\ra W})\ch(\ind(\D_{N_2})) \ .
\end{align*}

The last equality follows from the product formula, Prop. \ref{prodhigh}. Since $N_2$ has a nonvanishing higher signature, the term $\ch(\ind(\D_{N_2}))$ is nonzero. Furthermore, $\rho(F|_{\ra W}) \neq 0$ since $\rho_{(2)}(\ra W)-\rho_{(2)}(\ra V) \neq 0$.
\end{proof}

Our results should also yield the appropriate framework for transfering the methods in \cite[\S 2B]{pstor} from the positive scalar curvature case to surgery theory. That such a transfer should be possible was pointed out in \cite[p. 358]{pstor}. Thus, for a group $\Gamma$ with torsion one expects to detect a free abelian subgroup of higher rank in $L_{2k}(\bbbz\Gamma)$ acting freely on the structure set. The rank depends on the number of conjugacy classes of elements of finite order in $\Gamma$. Since our methods require to impose growth conditions on the conjugacy classes (as in \cite{pstor}), the result would still be less general than the predictions alluded to in the Problem on \cite[p. 317]{cw}.

\textsc{Leibniz-Archiv\\
Waterloostr. 8\\
30169 Hannover\\
Germany} 

\textsc{Email: wahlcharlotte@googlemail.com}

\end{document}